\documentclass{amsart}

\usepackage{amssymb}

\usepackage{graphicx}

\usepackage{amsmath, amsfonts, mathtools, tikz,lipsum}

\usepackage{etoolbox}
\newcommand{\zerodisplayskips}{%
  \setlength{\abovedisplayskip}{1pt}%
  \setlength{\belowdisplayskip}{1pt}%
  \setlength{\abovedisplayshortskip}{1pt}%
  \setlength{\belowdisplayshortskip}{1pt}}
\appto{\normalsize}{\zerodisplayskips}
\appto{\small}{\zerodisplayskips}
\appto{\footnotesize}{\zerodisplayskips}

\newtheorem{theorem}{Theorem}[section]
\newtheorem{lemma}[theorem]{Lemma}
\newtheorem{prop}[theorem]{Proposition}

\theoremstyle{definition}
\newtheorem{definition}[theorem]{Definition}

\theoremstyle{remark}

\theoremstyle{corollary}

\numberwithin{equation}{section}

\newcommand{\veps}{\varepsilon}

\newcommand{\mZ}{\mathbb{Z}}

\allowdisplaybreaks
\def\ge{\mathfrak{g}}

\def\te{\mathfrak{t}}
\def\al{\alpha}
\def\cd{\cdots}
\def\cV{\mathcal{V}}

\def\C{\mathbb{C}}
\def\Z{\mathbb{Z}}
\def\la{\lambda}
\def\L{\Lambda}

\begin{document}

\title[$D_6^{(1)}$- Geometric Crystal]{$D_6^{(1)}$- Geometric Crystal at the spin node}


\author{Kailash C. Misra}
\address{Department of Mathematics,
North Carolina State University, Raleigh, NC 27695-8205, USA}
\email{misra@ncsu.edu}
\thanks{KCM is partially supported by the Simons Foundation Grant \#307555.}

\author{Suchada Pongprasert}
\address{Department of Mathematics,
Srinakharinwirot University, Bangkok, Thailand 10310}
\email{suchadapo@g.swu.ac.th }



\date{}

\begin{abstract}
Let $\ge$ be an affine Lie algebra with index set $I = \{0, 1, 2, \cdots , n\}$. It is conjectured that  for each Dynkin node $k \in I \setminus \{0\}$ the affine Lie algebra $\ge$ has a positive geometric crystal.  In this paper we construct a positive geometric crystal for the affine Lie algebra $D_6^{(1)}$ corresponding to the Dynkin spin node $k= 6$. 
\end{abstract}

\maketitle

\section{Introduction}
Let $\ge$ be an affine Lie algebra \cite{Kac}  with Cartan datum $\{A, \Pi, \check{\Pi}, P, \check{P}\}$ and index set $I = \{0,1, \cdots , n\}$ where $A= (a_{ij})_{i,j \in I}$ is the affine GCM, $\Pi = \{\alpha_i \mid i \in I\}$ is the set of simple roots, $\check{\Pi} = \{\check{\alpha_i} \mid i \in I\}$ is the set of simple coroots, $P$ and $\check{P}$ are the weight lattice, and coweight lattice respectively. Let $\te = \C \otimes_{\mZ}\check{P}$, $\bf c$,  $\delta$, and  $\{\L_i \mid i\in I\}$ denote the Cartan subalgebra, the canonical central element, the null root and the set of fundamental weights respectively. Note that $\alpha_j(\check{\alpha}_i) = a_{ij}$ and $\L_j(\check{\alpha_i}) = \delta_{ij}$ and $\te = \text{span}_{\C}\{\check{\alpha}_i, d \mid i \in I\}$ where $d$ is a degree derivation. Then $P = \oplus_{j\in I}\mZ\L_j \oplus \mZ\delta \subset \te^*$, $\check{P} = \oplus_{i\in I}\mZ\check{\alpha}_i \oplus \mZ d \subset \te$ and $P_{cl} = P/\mZ\delta$ is called the classical weight lattice. The set $P^+ = \{\lambda \in P\mid \lambda (\check{\alpha}_i) \in \mZ_{\geq 0} \; \text{for} \; \text{all} \; i \in I\}$ (resp. $P_{cl}^+ = \{\lambda \in P_{cl} \mid \lambda (\check{\alpha}_i) \in \mZ_{\geq 0} \; \text{for} \; \text{all} \; i \in I\}$) is called the set of (affine) dominant (resp. classical dominant) weights and we say that $\lambda \in P^+$ or $P_{cl}^+$ has level $l = \lambda ({\bf c})$. We denote $(P^+)_l$ (resp. $(P_{cl}^+)_l$) to be the set of affine (resp. classical) dominant weights of level $l$. We denote $\te_{cl}^* = \te^*/{{\mathbb C}\delta}$ and 
$(\te_{cl}^*)_0 = \{\la \in \te_{cl}^* \mid \langle {\bf c}, \la\rangle = 0\}$. We also denote $\ge_i$ to be the subalgebra of $\ge$ with index set  $I_i = I \setminus \{i\}$ which is a finite dimensional semisimple  Lie algebra. The Weyl group $W$ of $\ge$ is generated by the simple reflections $\{s_i \mid i \in I\}$. The sets $\Delta$, $\Delta_+$ ,  $\Delta^{re}:=\{w(\al_i)|w\in W,\,\,i\in I\}$ and $\Delta^{re}_+ = \Delta_+ \cap \Delta^{re}$ are called the set of roots, postive roots, real roots and positive real roots respectively. In this paper we will assume $\ge$ to be simply laced which implies that the affine GCM $A = (a_{ij})_{i,j \in I}$ is symmetric.

Let $\ge'$ be the derived Lie algebra of $\ge$ and let  $G$ be the Kac-Moody group associated  with $\ge'$(\cite{KP, PK}).
Let $U_{\alpha}:=\exp\ge_{\alpha}$ $(\alpha\in \Delta^{re})$ be the one-parameter subgroup of $G$. The group $G$ is generated by 
$U_{\alpha}$ $(\alpha\in \Delta^{re})$. Let $U^{\pm}:=\langle U_{\pm\alpha}|\alpha\in\Delta^{re}_+\rangle$ be the subgroup generated by $U_{\pm\alpha}$
($\al\in \Delta^{re}_+$).
For any $i\in I$, there exists a unique homomorphism;
$\phi_i:SL_2(\C)\rightarrow G$ such that
\[
\hspace{-2pt}\phi_i\left(
\left(
\begin{array}{cc}
c&0\\
0&c^{-1}
\end{array}
\right)\right)=c^{\check{\alpha}_i},\,
\phi_i\left(
\left(
\begin{array}{cc}
1&t\\
0&1
\end{array}
\right)\right)=\exp(t e_i),\,
 \phi_i\left(
\left(
\begin{array}{cc}
1&0\\
t&1
\end{array}
\right)\right)=\exp(t f_i).
\]
where $c\in\C^\times$ and $t\in\C$.
Set $\check{\alpha}_i(c):=c^{\check{\alpha}_i}$,
$x_i(t):=\exp{(t e_i)}$, $y_i(t):=\exp{(t f_i)}$, 
$G_i:=\phi_i(SL_2(\C))$,
$H_i:=\phi_i(\{{\rm diag}(c,c^{-1})\mid 
c\in\C \setminus \{0\}\})$. 
Let $H$  be the subgroup of $G$ generated by $H_i$'s
with the Lie algebra $\te$. 
Then $H$ is called a  maximal torus in $G$, and 
$B^{\pm}=U^{\pm}H$ are the Borel subgroups of $G$.
The element $\bar{s}_i:=x_i(-1)y_i(1)x_i(-1) \in N_G(H)$  is a representative of 
$s_i\in W=N_G(H)/H$. 

The geometric crystal for the simply laced affine Lie algebra $\ge$ is defined as follows.
\begin{definition}\label{geometric}(\cite{BK},\cite{N}) 
The geometric crystal for the simply laced affine Lie algebra $\ge$ is a quadruple $\cV(\ge)=(X, \{e_i\}_{i \in I}, \{\gamma_i\}_{i \in I},$ 
$\{\veps_i\}_{i\in I})$, 
where $X$ is an ind-variety,  $e_i:\C^\times\times
X\longrightarrow X$ $((c,x)\mapsto e^c_i(x))$
are rational $\C^\times$-actions and  
$\gamma_i,\veps_i:X\longrightarrow 
\C$ $(i\in I)$ are rational functions satisfying the following:
\begin{enumerate}
\item $\{1\}\times X\subset {\rm dom}(e_i) \;
{\rm for} \; {\rm any} \; i\in I.$\\
\item $\gamma_j(e^c_i(x))=c^{a_{ij}}\gamma_j(x).$\\
\item $ \{e_i\}_{i\in I} \; {\rm satisfy \; the \; following \; relations}:\\
 \begin{array}{lll}
&\hspace{-20pt} \quad e^{c_1}_{i}e^{c_2}_{j}
=e^{c_2}_{j}e^{c_1}_{i}&
{\rm if }\,\,a_{ij}=a_{ji}=0,\\
&\hspace{-20pt} \quad e^{c_1}_{i}e^{c_1c_2}_{j}e^{c_2}_{i}
=e^{c_2}_{j}e^{c_1c_2}_{i}e^{c_1}_{j}&
{\rm if }\,\,a_{ij}=a_{ji}=-1,\\
\end{array}$\\
\item $\veps_i(e_i^c(x))=c^{-1}\veps_i(x)$ and $\veps_i(e_j^c(x))=\veps_i(x) \qquad {\rm if }\,
a_{i,j}=a_{j,i}=0.$
\end{enumerate}
\end{definition}
For fixed $i \in I$, let $G^i$ be the reductive algebraic group with Lie algebra $\ge_i$ and $B^i$, $W^i$ be its Borel subgroup, Weyl group respectively.
We consider the flag variety $X^i:=G^i/{B^i}$. For $w \in W^i$, the Schubert cell $X^i_w$ associated with $w$ has a natural $\ge_i$-geometric crystal structure \cite{BK, N}. Let $w=s_{i_1} s_{i_2} \cdots s_{i_l}$ be a reduced expression. For ${\bf i}:=(i_1, i_2, \cdots ,i_l)$, set 
\begin{equation*}
B_{\bf i}^-
:=\{Y_{\bf i}(c_1, c_2, \cdots ,c_l)
:=Y_{i_1}(c_1) Y_{i_2}(c_2)\cdots Y_{i_l}(c_l)
\,\vert\, c_1, c_2, \cdots ,c_l\in\C^\times\}\subset B^-,
\label{bw1}
\end{equation*}
where $Y_j(c):=y_j(\frac{1}{c}){\check\al}_j(c) = y_j(\frac{1}{c})c^{{\check\al}_j}$. Then  we have the following result.
\begin{theorem}\label{schubert}\cite{BK, N}
The set $B_{\bf i}^-$ with the explicit actions of \; $e^c_k$, $\veps_k$, and $\gamma_k$ , for $k \in I_i, c \in \C^\times$ given by:
\begin{eqnarray}
&& e_k^c(Y_{\bf i}(c_1,\cdots,c_l))
=Y_{\bf i}({\mathcal C}_1,\cdots,{\mathcal C}_l)), \notag\\
&&\text{where} \notag\\
&&{\mathcal C}_j:=
c_j\cdot \frac{\displaystyle \sum_{1\leq m\leq j,i_m=k}
 \frac{c}
{c_1^{a_{i_1,k}}\cdots c_{m-1}^{a_{i_{m-1},k}}c_m}
+\sum_{j< m\leq k,i_m=k} \frac{1}
{c_1^{a_{i_1,k}}\cdots c_{m-1}^{a_{i_{m-1},k}}c_m}}
{\displaystyle\sum_{1\leq m<j,i_m=k} 
 \frac{c}{c_1^{a_{i_1,k}}\cdots c_{m-1}^{a_{i_{m-1},k}}c_m}+ 
\mathop\sum_{j\leq m\leq k,i_m=k}  \frac{1}
{c_1^{a_{i_1,k}}\cdots c_{m-1}^{a_{i_{m-1},k}}c_m}},\\
&& \veps_k(Y_{\bf i}(c_1,\cd,c_l))=
\sum_{1\leq m\leq l,i_m=k} \frac{1}
{c_1^{a_{i_1,k}}\cdots c_{m-1}^{a_{i_{m-1},k}}c_m},\\
&&\gamma_k(Y_{\bf i}(c_1,\cdots,c_l))
=c_1^{a_{i_1,k}}\cdots c_l^{a_{i_l,k}},
\end{eqnarray}
is a geometric crystal isomorphic to $X^i_w$.
\end{theorem}

The geometric crystal $\cV(\ge)=(X, \{e_i\}_{i \in I}, 
\{\gamma_i\}_{i \in I}, \{\veps_i\}_{i\in I})$ is said to be positive if it has a 
positive structure \cite{BK, KNO, N}. 
Roughly speaking this means that each of the rational maps 
$e^c_i$, $\veps_i$  and $\gamma_i$ are
ratios of polynomial functions with positive coefficients. 
For example, $B_{\bf i}^-$ is a positive geometric crystal.

It was conjectured in \cite{KNO} that for each affine Lie algebra $\ge$ and 
each Dynkin index $k \in I \setminus \{0\}$, there exists a positive geometric crystal
$\cV(\ge)=(X, \{e_i\}_{i \in I}, \{\gamma_i\}_{i \in I}, \\
\{\veps_i\}_{i\in I})$. So far positive geometric crystals have been explicitly constructed for $\ge = A_n^{(1)}, 
B_n^{(1)}, C_n^{(1)}, D_n^{(1)}, A_{2n-1}^{(2)}, A_{2n}^{(2)},$
$D_{n+1}^{(2)}$ \cite{KNO}, $\ge = D_4^{(3)}$ \cite{IN}, $\ge = G_2^{(1)}$ \cite{N3} for $k=1$, $\ge = A_n^{(1)}$ \cite{MN1, MN2} for $k > 1$ , 
and $\ge = D_5^{(1)}$ for $k=5$ \cite{IMP}.
In this paper we construct explicitly positive geometric crystal  for $\ge = D_6^{(1)}$ corresponding to the  Dynkin spin node $k = 6$ in the level zero fundamental spin module $W(\varpi_6)$. It turns out that the diagram automorphism $\sigma$ in this case has order two, unlike the corresponding diagram automorphism in $D_5^{(1)}$-case \cite{IMP}.

\section{The Quantum Affine algebra $U_q(D_6^{(1)})$}

From now on we assume $\ge$ to be the affine Lie algebra $D_6^{(1)}$ with index set $I = \{0,1,2,3,4,5,6\}$, Cartan matrix $A = (a_{ij})_{i,j \in I}$ where $a_{ii} = 2, a_{j,j + 1} = -1 = a_{j+1,j}, \; j = 1,2,3,4, a_{02} = a_{20} = a_{46} = a_{64} = -1, a_{ij} = 0$ otherwise and Dynkin diagram:
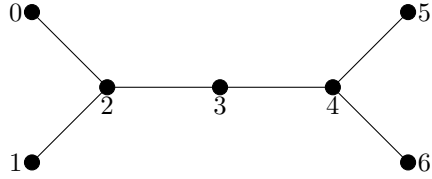
\begin{center}
\begin{tikzpicture}
\draw (-2,1)--(-1,0); \draw (-2,-1)--(-1,0); \draw (-1,0)--(.5,0); \draw (.5,0)--(2,0); \draw (2,0)--(3,1); \draw (2,0)--(3,-1);
\draw [fill] (-2,1) circle [radius=0.1] node[left=.1pt] (a) {0};
\draw [fill] (-2,-1) circle [radius=0.1] node[left=.1pt] (b) {1};
\draw [fill] (-1,0) circle [radius=0.1] node[below=.3pt] (c) {2};
\draw [fill] (.5,0) circle [radius=0.1] node[below=.3pt] (d) {3};
\draw [fill] (2,0) circle [radius=0.1] node[below=.3pt] (e) {4};
\draw [fill] (3,1) circle [radius=0.1] node[right=.1pt] (f) {5};
\draw [fill] (3,-1) circle [radius=0.1] node[right=.1pt] (g) {6};        
\end{tikzpicture}
\end{center}
Let $\{\alpha_0, \alpha_1, \alpha_2, \alpha_3, \alpha_4, \alpha_5, \alpha_6\}, \ \{\check{\alpha_0}, \check{\alpha_1}, \check{\alpha_2}, \check{\alpha_3}, \check{\alpha_4}, \check{\alpha_5}, \check{\alpha_6}\}$ and $\{\Lambda_0, \Lambda_1, \Lambda_2, \Lambda_3, \Lambda_4, \\\Lambda_5, \Lambda_6\}$ denote the set of simple roots, simple coroots and fundamental weights, respectively.
Then ${\bf c} =\check{\alpha_0}+\check{\alpha_1}+2\check{\alpha_2}+2\check{\alpha_3}+2\check{\alpha_4}+\check{\alpha_5}+\check{\alpha_6}$ and $\delta = \al_0 +\al_1+2\al_2+2\al_3+2\al_4+\al_5+\al_6$ are the canonical central element and null root respectively. The sets $P_{cl} = \oplus_{j=0}^6 \Z\L_j$ and $P = P_{cl}\oplus\Z\delta$ are called classical weight lattice and weight lattice respectively.

We consider the Dynkin diagram automorphism 
$\sigma$ defined by 
$$\sigma : 0 \mapsto 6, 1 \mapsto 5, 2 \mapsto 4, 3 \mapsto 3, 4 \mapsto 2, 5 \mapsto 1, 6 \mapsto 0.$$

\medskip

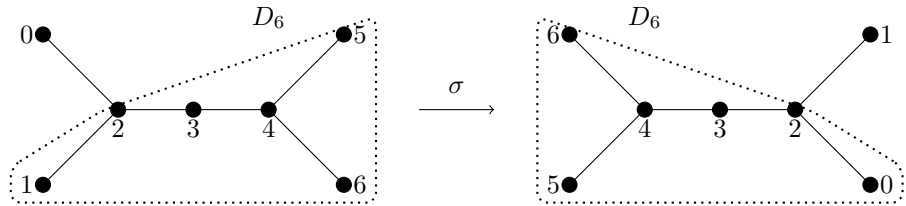
\begin{tikzpicture}
\draw (-8,1)--(-7,0); \draw (-8,-1)--(-7,0); \draw (-7,0)--(-6,0); \draw (-6,0)--(-5,0); \draw (-5,0)--(-4,1); \draw (-5,0)--(-4,-1);
\draw [fill] (-8,1) circle [radius=0.1] node[left=.1pt] (a) {0};
\draw [fill] (-8,-1) circle [radius=0.1] node[left=.1pt] (b) {1};
\draw [fill] (-7,0) circle [radius=0.1] node[below=.3pt] (c) {2};
\draw [fill] (-6,0) circle [radius=0.1] node[below=.3pt] (d) {3};
\draw [fill] (-5,0) circle [radius=0.1] node[below=.3pt] (e) {4};
\draw [fill] (-4,1) circle [radius=0.1] node[right=.1pt] (f) {5};
\draw [fill] (-4,-1) circle [radius=0.1] node[right=.1pt] (g) {6};

\path [draw, dotted,thick,rounded corners] 
               (b.north west) 
            -- (c.north west) 
            -- (f.north east) 
            -- (g.south east) 
            -- (b.south west) 
            -- cycle ;   
\draw [fill=white] (-5,1.5) node[below=.1pt] {$D_6$};            
 
\draw [->](-3,0)--(-2,0) ;
\draw [fill=white] (-2.5,.5) node[below=.1pt, black] {$\sigma$};

\draw (3,1)--(2,0); \draw (3,-1)--(2,0); \draw (2,0)--(1,0); \draw (1,0)--(0,0); \draw (0,0)--(-1,1); \draw (0,0)--(-1,-1);
\draw [fill] (3,1) circle [radius=0.1] node[right=.1pt] (h) {1};
\draw [fill] (3,-1) circle [radius=0.1] node[right=.1pt] (i) {0};
\draw [fill] (2,0) circle [radius=0.1] node[below=.3pt] (j) {2};
\draw [fill] (1,0) circle [radius=0.1] node[below=.3pt] (k) {3};
\draw [fill] (0,0) circle [radius=0.1] node[below=.3pt] (l) {4};
\draw [fill] (-1,-1) circle [radius=0.1] node[left=.1pt] (m) {5};
\draw [fill] (-1,1) circle [radius=0.1] node[left=.1pt] (n) {6};

\path [draw, dotted,thick, rounded corners] 
             (i.north east) 
            -- (j.north east) 
            -- (n.north west) 
            -- (m.south west) 
            -- (i.south east) 
            -- cycle ;   
\draw [fill=white] (0,1.5) node[below=.1pt] {$D_6$};    
\end{tikzpicture}
\medskip

Let $\ge_j$ (resp. $\sigma(\ge)_j)$) be the subalgebra of $\ge$ 
(resp. $\sigma(\ge)$) with index set $I_j= I \setminus \{j\}$. Then observe that $\ge_0$ as well as $\ge_1$ and $\sigma(\ge)_1$ are isomorphic to $D_6$.

Let $W(\varpi_6)$ be the level $0$ fundamental $U'_q(\mathfrak{g})$-module associated with the level $0$ weight $\varpi_6 = \L_6 - \L_0$  \cite{K0}. By [\cite{K0}, Theorem 5.17], $W(\varpi_6)$ is a finite-dimensional irreducible integrable $U'_q(\mathfrak{g})$-module and has a global basis with a simple crystal. Thus, we can consider the specialization $q=1$ and obtain the finite-dimensional $D_6^{(1)}$-module $W(\varpi_6)$, which we call the fundamental $D_6^{(1)}$- module  and use the same notation as above. Below we give the explicit description of $W(\varpi_6)$.

\section{Fundamental Representation {\bf{$W(\varpi_6)$}} for $D_6^{(1)}$} 
The fundamental $D_6^{(1)}$-module $W(\varpi_6)$ is a 32-dimensional module with the basis 
$$\{v= (i_1, i_2, i_3, i_4, i_5, i_6) |i_j \in \{+, -\}, \ i_1 i_2 i_3 i_4 i_5 i_6 = +\}.$$
The actions of the generators $e_k$, and $f_k$, $0\le k \le 6$ of $D_6^{(1)}$, on the basis vectors are given as follows.
\begin{displaymath}
\begin{split} 
& e_k (i_1, i_2, i_3, i_4, i_5, i_6) = \left \{
	\begin{array}{lllll}
	(-, -, i_3, i_4, i_5, i_6)    \hspace{.2cm}& \text{ if  }  k = 0, \ (i_1, i_2)=(+,+)\\
	(i_1, \ldots, +, -,\ldots, i_6)  &\text{ if  } k \neq 0, \ k \neq 6, \ \\
	\hspace{1.2cm}k \hspace{.28cm}k+1 &  \hspace{10pt} (i_k, i_{k+1})=(-,+)\\
	(i_1, i_2, i_3, i_4,+, +)    \hspace{.2cm}& \text{ if  } k = 6, \ (i_5, i_6)=(-,-)\\
	0 \hspace{1cm} &\text{ otherwise.}
	\end{array}
	\right. \\
& f_k (i_1, i_2, i_3, i_4, i_5,i_6) = \left \{
	\begin{array}{lllll}
	(+, +, i_3, i_4, i_5,i_6)    \hspace{.2cm}& \text{ if  }  k = 0, \ (i_1, i_2)=(-,-)\\
	(i_1, \ldots, -, +,\ldots, i_6)  &\text{ if  } k \neq 0, \ k \neq 6, \ \\
	\hspace{1.2cm}k \hspace{.28cm}k+1 & \hspace{10pt} (i_k, i_{k+1})=(+,-) \\
	(i_1, i_2, i_3, i_4,-, -)    \hspace{.2cm} & \text{ if  } k = 6, \ (i_5, i_6)=(+,+)\\
	0 \hspace{1cm}  &\text{ otherwise.}
	\end{array}
	\right.
\end{split}
\end{displaymath}

Furthermore, we observe that
\begin{displaymath}
\begin{split} 
& \langle {\check\alpha_k}, {\rm wt}(v)\rangle = \left \{
\begin{array}{lllll}
1   \hspace{.1cm}& & \text{ if  } \ k = 0, \ (i_1, i_2)=(-,-)\\
& & \text{ or  }  k \neq 0, \ k \neq 6, \ (i_k, i_{k+1})=(+,-)\\
& & \text{ or  }  k = 6, (i_5, i_6)=(+,+)\\
-1 & & \text{ if  } \ k = 0, \ (i_1, i_2)=(+,+)\\
& & \text{ or  }  k \neq 0, \ k \neq 6, (i_k, i_{k+1})=(-,+)\\
& & \text{ or  }  k = 6, \ (i_5, i_6)=(-,-)\\
0 & & \text{ otherwise.}
\end{array}
\right. \\
\end{split}
\end{displaymath}

Note that in $W(\varpi_6)$, we have $(+,+,+,+,+,+)$  (resp. $(-,+,+,+,+,-)$) is a $\mathfrak{g}_0$ (resp. $\mathfrak{g}_1$) highest weight vector with weight $\varpi_6 = \Lambda_6-\Lambda_0$ (resp. $\check{\varpi_6}:=\Lambda_5-\Lambda_1$). We define $\sigma(\L_j) = \L_{\sigma(j)}$ for $j \in I$. Then we define the action of $\sigma$ on $W(\varpi_6)$ by $\sigma(v) = v'$ if $\sigma({\rm wt}(v)) = {\rm wt}(v')$. 

\section{Positive Geometric Crystals in \bf{$W(\varpi_6)$}}
For $\xi \in (\mathfrak{t}^*_{\text{cl}})_0$, let $t(\xi)$ be the translation as in [\cite{K0}, Sect 4]. Define simple reflections $s_k (\lambda) := \lambda - \lambda({\check\alpha_k})\alpha_k , k \in I$ and let $W = \langle s_k \mid k \in I \rangle$ be the Weyl group for $D_6^{(1)}$. Then we have
\begin{align*}
&t(\varpi_6) = \sigma s_6s_4s_3s_2s_5s_4s_3s_6s_4s_5s_1s_2s_3s_4s_6=\sigma w_1,\\
&t(\check{\varpi_6}) = \sigma s_5s_4s_3s_2s_6s_4s_3s_5s_4s_6s_0s_2s_3s_4s_5=\sigma w_2,
\end{align*}
where $w_1 = s_6s_4s_3s_2s_5s_4s_3s_6s_4s_5s_1s_2s_3s_4s_6 \in W^0$ and $w_2 = s_5s_4s_3s_2s_6s_4s_3s_5s_4s_6\\s_0s_2s_3s_4s_5 \in W^1$.
Associated with these Weyl group elements $w_1, w_2 \in W$, we define algebraic varieties $\mathcal{V}_1, \mathcal{V}_2 \subset W(\varpi_6)$ as follows.
\begin{align*}
\mathcal{V}_1	&=\big\{V_1(x):=Y_6(x_6^{(3)})Y_4(x_4^{(4)})Y_3(x_3^{(3)})Y_2(x_2^{(2)})Y_5(x_5^{(2)})Y_4(x_4^{(3)})Y_3(x_3^{(2)})Y_6(x_6^{(2)})\\
&\hspace{.7cm}Y_4(x_4^{(2)})Y_5(x_5^{(1)})Y_1(x_1^{(1)})Y_2(x_2^{(1)})Y_3(x_3^{(1)})Y_4(x_4^{(1)})Y_6(x_6^{(1)})(+,+,+,+,+,+) \mid\\
&\hspace{.7cm} x_m^{(l)} \in \mathbb{C}^{\times} \big\}, \\
\mathcal{V}_2	&=\big\{V_2(y):=Y_5(y_5^{(3)})Y_4(y_4^{(4)})Y_3(y_3^{(3)})Y_2(y_2^{(2)})Y_6(y_6^{(2)})Y_4(y_4^{(3)})Y_3(y_3^{(2)})Y_5(y_5^{(2)})\\
&\hspace{.7cm}Y_4(y_4^{(2)})Y_6(y_6^{(1)})Y_0(y_0^{(1)})Y_2(y_2^{(1)})Y_3(y_3^{(1)})Y_4(y_4^{(1)})Y_5(y_5^{(1)})(-,+,+,+,+,-) \mid \\
&\hspace{.7cm} y_m^{(l)} \in \mathbb{C}^{\times} \big\}, 
\end{align*}
where $x=(x_6^{(3)},x_4^{(4)}, x_3^{(3)},x_2^{(2)}, x_5^{(2)}, x_4^{(3)},x_3^{(2)},x_6^{(2)}, x_4^{(2)}, x_5^{(1)}, x_1^{(1)}, x_2^{(1)},x_3^{(1)}, x_4^{(1)}, x_6^{(1)})$ and $y=(y_5^{(3)},y_4^{(4)},y_3^{(3)},y_2^{(2)},y_6^{(2)},y_4^{(3)}, y_3^{(2)},y_5^{(2)},y_4^{(2)},y_6^{(1)},y_0^{(1)},y_2^{(1)}, y_3^{(1)},y_4^{(1)},y_5^{(1)})$.

From the explicit actions of $f_k$'s on $W(\varpi_6)$, we observe that $f_k^2 =0$, for all $k \in I$. Therefore, we have
$$Y_k(c)=(1+\frac{f_k}{c}){\check\alpha_k}(c) = (1+\frac{f_k}{c})c^{\check\alpha_k} \ \text{for all} \ k \in I.$$ 
Thus we have the explicit forms of $V_1(x)$ and $V_2(y)$ as follows. \\
$V_1 (x) = x_6^{(3)}x_6^{(2)}x_6^{(1)}(+,+,+,+,+,+) + 
\big(x_6^{(2)}x_6^{(1)}  + \frac{x_4^{(4)}x_4^{(3)}x_6^{(1)}}{x_6^{(3)}} + \frac{x_4^{(4)}x_4^{(3)}x_4^{(2)}x_4^{(1)}}{x_6^{(3)}x_6^{(2)}}\big) \newline
\times (+,+,+,+,-,-) + \big(x_4^{(3)}x_6^{(1)} + \frac{x_3^{(3)}x_5^{(2)}x_6^{(1)}}{x_4^{(4)}} + \frac{x_4^{(3)}x_4^{(2)}x_4^{(1)}}{x_6^{(2)}} + \frac{x_3^{(3)}x_5^{(2)}x_3^{(2)}x_4^{(1)}}{x_4^{(4)}x_4^{(3)}} + 	\frac{x_3^{(3)}x_5^{(2)}x_4^{(2)}x_4^{(1)}}{x_4^{(4)}x_6^{(2)}}\newline
+ \frac{x_3^{(3)}x_5^{(2)}x_3^{(2)}x_5^{(1)}x_3^{(1)}}{x_4^{(4)}x_4^{(3)}x_4^{(2)}}\big) (+,+,+,-,+,-)+  \big(x_5^{(2)}x_6^{(1)} + \frac{x_2^{(2)}x_5^{(2)}x_4^{(1)}}{x_3^{(3)}} + \frac{x_5^{(2)}x_3^{(2)}x_4^{(1)}}{x_4^{(3)}} \newline
+ \frac{x_5^{(2)}x_4^{(2)}x_4^{(1)}}{x_6^{(2)}} + \frac{x_2^{(2)}x_5^{(2)}x_5^{(1)}x_2^{(1)}}{x_3^{(3)}x_3^{(2)}} + \frac{x_2^{(2)}x_5^{(2)}x_5^{(1)}x_3^{(1)}}{x_3^{(3)}x_4^{(2)}}+ \frac{x_5^{(2)}x_3^{(2)}x_5^{(1)}x_3^{(1)}}{x_4^{(3)}x_4^{(2)}}\big)(+,+,-,+,+,-) \newline
+\big(x_3^{(3)}x_6^{(1)} + \frac{x_3^{(3)}x_3^{(2)}x_3^{(1)}}{x_5^{(2)}} + \frac{x_3^{(3)}x_3^{(2)}x_4^{(1)}}{x_4^{(3)}} + \frac{x_3^{(3)}x_4^{(2)}x_4^{(1)}}{x_6^{(2)}} + \frac{x_3^{(3)}x_3^{(2)}x_5^{(1)}x_3^{(1)}}{x_4^{(3)}x_4^{(2)}}\big)(+,+,+,-,-,+)\newline
+\big(x_5^{(2)}x_4^{(1)} + \frac{x_5^{(2)}x_5^{(1)}x_1^{(1)}}{x_2^{(2)}} + \frac{x_5^{(2)}x_5^{(1)}x_2^{(1)}}{x_3^{(2)}} + \frac{x_5^{(2)}x_5^{(1)}x_3^{(1)}}{x_4^{(2)}} \big)(+,-,+,+,+,-) + \big(x_4^{(4)}x_6^{(1)} \newline
+ \frac{x_4^{(4)}x_2^{(2)}x_4^{(1)}}{x_3^{(3)}} + \frac{x_4^{(4)}x_3^{(2)}x_3^{(1)}}{x_5^{(2)}} + \frac{x_4^{(4)}x_3^{(2)}x_4^{(1)}}{x_4^{(3)}} +\frac{x_4^{(4)}x_4^{(2)}x_4^{(1)}}{x_6^{(2)}} + \frac{x_4^{(4)}x_2^{(2)}x_4^{(3)}x_3^{(1)}}{x_3^{(3)}x_5^{(2)}}+ \frac{x_4^{(4)}x_2^{(2)}x_5^{(1)}x_2^{(1)}}{x_3^{(3)}x_3^{(2)}} \newline
+ 		\frac{x_4^{(4)}x_2^{(2)}x_5^{(1)}x_3^{(1)}}	{x_3^{(3)}x_4^{(2)}} + \frac{x_4^{(4)}x_3^{(2)}x_5^{(1)}x_3^{(1)}}{x_4^{(3)}x_4^{(2)}} + \frac{x_4^{(4)}x_2^{(2)}x_4^{(3)}x_4^{(2)}		x_2^{(1)}}{x_3^{(3)}x_5^{(2)}x_3^{(2)}}\big) (+,+,-,+,-,+) \newline
+x_5^{(2)}x_5^{(1)}(-,+,+,+,+,-) + \big(x_4^{(4)}x_4^{(1)} + \frac{x_4^{(4)}x_5^{(1)}x_1^{(1)}}{x_2^{(2)}} + \frac{x_4^{(4)}x_4^{(3)}x_3^{(1)}}{x_5^{(2)}} + \frac{x_4^{(4)}x_5^{(1)}x_2^{(1)}}{x_3^{(2)}} \newline
+ \frac{x_4^{(4)}x_5^{(1)}x_3^{(1)}}{x_4^{(2)}} + \frac{x_4^{(4)}x_4^{(3)}x_4^{(2)}x_1^{(1)}}{x_2^{(2)}x_5^{(2)}}+ \frac{x_4^{(4)}x_4^{(3)}x_4^{(2)}x_2^{(1)}}		{x_5^{(2)}x_3^{(2)}} \big)(+,-,+,+,-,+) + \big(x_6^{(3)}x_6^{(1)} + \frac{x_6^{(3)}x_2^{(2)}x_3^{(1)}}{x_4^{(4)}} \newline
+ \frac{x_6^{(3)}x_2^{(2)}x_4^{(1)}}{x_3^{(3)}} + \frac{x_6^{(3)}		x_3^{(2)}x_3^{(1)}}{x_5^{(2)}} +	\frac{x_6^{(3)}x_3^{(2)}x_4^{(1)}}{x_4^{(3)}} + \frac{x_6^{(3)}x_4^{(2)}x_4^{(1)}}{x_6^{(2)}} + \frac{x_6^{(3)}x_2^{(2)}x_6^{(2)}x_2^{(1)}}{x_4^{(4)}	x_4^{(3)}} + \frac{x_6^{(3)}x_2^{(2)}x_4^{(2)}x_2^{(1)}}{x_4^{(4)}x_3^{(2)}} \newline
+ \frac{x_6^{(3)}x_2^{(2)}x_4^{(3)}x_3^{(1)}}{x_3^{(3)}x_5^{(2)}} + \frac{x_6^{(3)}x_2^{(2)}x_5^{(1)}		x_2^{(1)}}{x_3^{(3)}x_3^{(2)}} + \frac{x_6^{(3)}x_2^{(2)}x_5^{(1)}x_3^{(1)}}{x_3^{(3)}x_4^{(2)}} + \frac{x_6^{(3)}x_3^{(2)}x_5^{(1)}x_3^{(1)}}{x_4^{(3)}x_4^{(2)}} + \frac{x_6^{(3)}		x_2^{(2)}x_4^{(3)}x_4^{(2)}x_2^{(1)}}{x_3^{(3)}x_5^{(2)}x_3^{(2)}}\big) \newline
\times (+,+,-,-,+,+)+ \big(x_4^{(4)}x_5^{(1)}  + \frac{x_4^{(4)}x_4^{(3)}x_4^{(2)}}{x_5^{(2)}}\big)(-,+,+,+,-,+) + \big(x_6^{(3)}x_4^{(1)} + \frac{x_6^{(3)}x_3^{(3)}x_3^{(1)}}{x_4^{(4)}} \newline
+ \frac{x_6^{(3)}x_5^{(1)}	x_1^{(1)}}{x_2^{(2)}} + \frac{x_6^{(3)}	x_4^{(3)}x_3^{(1)}}{x_5^{(2)}} +	\frac{x_6^{(3)}x_5^{(1)}x_2^{(1)}}{x_3^{(2)}} + \frac{x_6^{(3)}x_5^{(1)}x_3^{(1)}}{x_4^{(2)}} + \frac{x_6^{(3)}	x_3^{(3)}x_4^{(2)}x_1^{(1)}}{x_4^{(4)}x_2^{(2)}} + \frac{x_6^{(3)}x_3^{(3)}x_6^{(2)}x_2^{(1)}}{x_4^{(4)}x_4^{(3)}} \newline
+ \frac{x_6^{(3)}x_3^{(3)}x_4^{(2)}x_2^{(1)}}{x_4^{(4)}x_3^{(2)}} 	+ \frac{x_6^{(3)}x_4^{(3)}x_4^{(2)}x_1^{(1)}}{x_2^{(2)}x_5^{(2)}} + \frac{x_6^{(3)}x_4^{(3)}x_4^{(2)}x_2^{(1)}}{x_5^{(2)}x_3^{(2)}} + \frac{x_6^{(3)}x_3^{(3)}x_3^{(2)}x_6^{(2)}		x_1^{(1)}}{x_4^{(4)}x_2^{(2)}x_4^{(3)}}\big)(+,-,+,-,+,+) \newline
+ \big(x_6^{(1)} + \frac{x_2^{(2)}x_2^{(1)}}{x_6^{(3)}} + \frac{x_2^{(2)}x_3^{(1)}}{x_4^{(4)}} + \frac{x_2^{(2)}x_4^{(1)}}		{x_3^{(3)}} + \frac{x_3^{(2)}x_3^{(1)}}{x_5^{(2)}} + \frac{x_3^{(2)}x_4^{(1)}}{x_4^{(3)}} + \frac{x_4^{(2)}x_4^{(1)}}{x_6^{(2)}} + \frac{x_2^{(2)}x_6^{(2)}x_2^{(1)}}{x_4^{(4)}			x_4^{(3)}} \newline
+ \frac{x_2^{(2)}x_4^{(2)}x_2^{(1)}}{x_4^{(4)}x_3^{(2)}} + \frac{x_2^{(2)}x_4^{(3)}x_3^{(1)}}{x_3^{(3)}x_5^{(2)}} + \frac{x_2^{(2)}x_5^{(1)}x_2^{(1)}}{x_3^{(3)}x_3^{(2)}} 	+ \frac{x_2^{(2)}x_5^{(1)}x_3^{(1)}}{x_3^{(3)}x_4^{(2)}} + \frac{x_3^{(2)}x_5^{(1)}x_3^{(1)}}{x_4^{(3)}x_4^{(2)}} +\frac{x_2^{(2)}x_4^{(3)}x_4^{(2)}x_2^{(1)}}{x_3^{(3)}x_5^{(2)}x_3^{(2)}}\big)\newline
\times (+,+,-,-,-,-)+\big(x_6^{(3)}x_5^{(1)}  + \frac{x_6^{(3)}x_3^{(3)}x_4^{(2)}}{x_4^{(4)}} + \frac{x_6^{(3)}x_4^{(3)}x_4^{(2)}}{x_5^{(2)}} + \frac{x_6^{(3)}x_3^{(3)}x_3^{(2)}x_6^{(2)}}{x_4^{(4)}x_4^{(3)}}		\big)\newline
\times (-,+,+,-,+,+) + \big(x_6^{(3)}x_3^{(1)}  + \frac{x_6^{(3)}x_6^{(2)}x_1^{(1)}}{x_3^{(3)}} + \frac{x_6^{(3)}x_4^{(2)}x_1^{(1)}}{x_2^{(2)}} + \frac{x_6^{(3)}x_6^{(2)}			x_2^{(1)}}{x_4^{(3)}} + \frac{x_6^{(3)}x_4^{(2)}x_2^{(1)}}{x_3^{(2)}} \newline
+ \frac{x_6^{(3)}x_3^{(2)}x_6^{(2)}x_1^{(1)}}{x_2^{(2)}x_4^{(3)}}	\big) (+,-,-,+,+,+) + \big(x_4^{(1)} + 			\frac{x_3^{(3)}x_2^{(1)}}{x_6^{(3)}} + \frac{x_3^{(3)}x_3^{(1)}}{x_4^{(4)}} + \frac{x_5^{(1)}x_1^{(1)}}{x_2^{(2)}} + \frac{x_4^{(3)}x_3^{(1)}}{x_5^{(2)}} \newline
+ \frac{x_5^{(1)}x_2^{(1)}}		{x_3^{(2)}} + \frac{x_5^{(1)}x_3^{(1)}}{x_4^{(2)}} + \frac{x_3^{(3)}x_3^{(2)}x_1^{(1)}}{x_6^{(3)}x_2^{(2)}} + \frac{x_3^{(3)}x_4^{(2)}x_1^{(1)}}{x_4^{(4)}x_2^{(2)}} + \frac{x_3^{(3)}		x_6^{(2)}x_2^{(1)}}{x_4^{(4)}x_4^{(3)}} + \frac{x_3^{(3)}x_4^{(2)}x_2^{(1)}}{x_4^{(4)}x_3^{(2)}} + \frac{x_4^{(3)}x_4^{(2)}x_1^{(1)}}{x_2^{(2)}x_5^{(2)}} \newline
+ \frac{x_4^{(3)}x_4^{(2)}		x_2^{(1)}}	{x_5^{(2)}x_3^{(2)}} +\frac{x_3^{(3)}x_3^{(2)}x_6^{(2)}x_1^{(1)}}{x_4^{(4)}x_2^{(2)}x_4^{(3)}}\big)(+,-,+,-,-,-)+
 \big(x_6^{(3)}x_4^{(2)}  + \frac{x_6^{(3)}x_2^{(2)}x_6^{(2)}}{x_3^{(3)}} + \frac{x_6^{(3)}x_3^{(2)}x_6^{(2)}}{x_4^{(3)}} \big)\newline
 \times (-,+,-,+,+,+) + \big(x_5^{(1)} + \frac{x_3^{(3)}x_3^{(2)}}		{x_6^{(3)}} + \frac{x_3^{(3)}x_4^{(2)}}{x_4^{(4)}} + \frac{x_4^{(3)}x_4^{(2)}}	{x_5^{(2)}} + \frac{x_3^{(3)}x_3^{(2)}x_6^{(2)}}{x_4^{(4)}x_4^{(3)}} \big)(-,+,+,-,-,-) \newline 
 + \big(x_3^{(1)} + 		\frac{x_4^{(4)}x_2^{(1)}}{x_6^{(3)}} + \frac{x_6^{(2)}x_1^{(1)}}{x_3^{(3)}} + \frac{x_4^{(2)}x_1^{(1)}}{x_2^{(2)}} + \frac{x_6^{(2)}x_2^{(1)}}{x_4^{(3)}} + \frac{x_4^{(2)}x_2^{(1)}}		{x_3^{(2)}} + \frac{x_4^{(4)}x_4^{(3)}x_1^{(1)}}{x_6^{(3)}x_3^{(3)}} + \frac{x_4^{(4)}x_3^{(2)}x_1^{(1)}}{x_6^{(3)}x_2^{(2)}} \newline
 + \frac{x_3^{(2)}	x_6^{(2)}x_1^{(1)}}{x_2^{(2)}x_4^{(3)}} 	\big)(+,-,-,+,-,-)+
x_6^{(3)}x_6^{(2)}(-,-,+,+,+,+) + \big(x_4^{(2)} + \frac{x_4^{(4)}x_3^{(2)}}{x_6^{(3)}} + \frac{x_2^{(2)}x_6^{(2)}}{x_3^{(3)}} \newline
+ \frac{x_3^{(2)}x_6^{(2)}}{x_4^{(3)}} + \frac{x_4^{(4)}			x_2^{(2)}x_4^{(3)}}{x_6^{(3)}x_3^{(3)}} \big)(-,+,-,+,-,-) + \big(x_2^{(1)} + \frac{x_5^{(2)}x_1^{(1)}}{x_4^{(4)}} + \frac{x_4^{(3)}x_1^{(1)}}{x_3^{(3)}} + \frac{x_3^{(2)}x_1^{(1)}}		{x_2^{(2)}} \big)\newline
\times (+,-,-,-,+,-)+ \big(x_6^{(2)} + \frac{x_4^{(4)}x_4^{(3)}}{x_6^{(3)}} \big)(-,-,+,+,-,-) + \big(x_3^{(2)} + \frac{x_2^{(2)}x_5^{(2)}}{x_4^{(4)}} + \frac{x_2^{(2)}x_4^{(3)}}{x_3^{(3)}} \big)\newline
\times(-,+,-,-,+,-) + 			x_1^{(1)}(+,-,-,-,-,+)+
 \big(x_4^{(3)} + \frac{x_3^{(3)}x_5^{(2)}}{x_4^{(4)}} \big)(-,-,+,-,+,-) \newline
 + x_2^{(2)}(-,+,-,-,-,+)+
x_5^{(2)}(-,-,-,+,+,-)+x_3^{(3)}(-,-,+,-,-,+)\newline
+x_4^{(4)}(-,-,-,+,-,+)+
x_6^{(3)}(-,-,-,-,+,+)+
(-,-,-,-,-,-)$,\\

$V_2 (y) = y_5^{(3)}y_5^{(2)}y_5^{(1)}(-,+,+,+,+,-) + 
\big(y_5^{(2)}y_5^{(1)}  + \frac{y_4^{(4)}y_4^{(3)}y_5^{(1)}}{y_5^{(3)}} + \frac{y_4^{(4)}y_4^{(3)}y_4^{(2)}y_4^{(1)}}{y_5^{(3)}y_5^{(2)}}\big)\newline
\times (-,+,+,+,-,+) +  \big(y_4^{(3)}y_5^{(1)} + \frac{y_3^{(3)}y_6^{(2)}y_5^{(1)}}{y_4^{(4)}} + \frac{y_4^{(3)}y_4^{(2)}y_4^{(1)}}{y_5^{(2)}} + \frac{y_3^{(3)}y_6^{(2)}y_3^{(2)}y_4^{(1)}}{y_4^{(4)}y_4^{(3)}} + 	\frac{y_3^{(3)}y_6^{(2)}y_4^{(2)}y_4^{(1)}}{y_4^{(4)}y_5^{(2)}} \newline
+ \frac{y_3^{(3)}y_6^{(2)}y_3^{(2)}y_6^{(1)}y_3^{(1)}}{y_4^{(4)}y_4^{(3)}y_4^{(2)}}\big)(-,+,+,-,+,+)+
\big(y_6^{(2)}y_5^{(1)} + \frac{y_2^{(2)}y_6^{(2)}y_4^{(1)}}{y_3^{(3)}} + \frac{y_6^{(2)}y_3^{(2)}y_4^{(1)}}{y_4^{(3)}} + \frac{y_6^{(2)}y_4^{(2)}y_4^{(1)}}{y_5^{(2)}} \newline
+ \frac{y_2^{(2)}y_6^{(2)}y_6^{(1)}y_2^{(1)}}{y_3^{(3)}y_3^{(2)}} + \frac{y_2^{(2)}y_6^{(2)}y_6^{(1)}y_3^{(1)}}{y_3^{(3)}y_4^{(2)}}+ \frac{y_6^{(2)}y_3^{(2)}y_6^{(1)}y_3^{(1)}}		{y_4^{(3)}y_4^{(2)}}\big) (-,+,-,+,+,+)+\big(y_3^{(3)}y_5^{(1)} + \frac{y_3^{(3)}y_3^{(2)}y_3^{(1)}}{y_6^{(2)}}\newline
+ \frac{y_3^{(3)}y_3^{(2)}y_4^{(1)}}{y_4^{(3)}} + 				\frac{y_3^{(3)}y_4^{(2)}y_4^{(1)}}{y_5^{(2)}} + \frac{y_3^{(3)}y_3^{(2)}y_6^{(1)}y_3^{(1)}}{y_4^{(3)}y_4^{(2)}}\big)(-,+,+,-,-,-)+
\big(y_6^{(2)}y_4^{(1)} + \frac{y_6^{(2)}y_6^{(1)}y_0^{(1)}}{y_2^{(2)}} \newline
+ \frac{y_6^{(2)}y_6^{(1)}y_2^{(1)}}{y_3^{(2)}} + \frac{y_6^{(2)}y_6^{(1)}y_3^{(1)}}{y_4^{(2)}} \big)(-,-,+,+,+,+) + 		\big(y_4^{(4)}y_5^{(1)} + \frac{y_4^{(4)}y_2^{(2)}y_4^{(1)}}{y_3^{(3)}} + \frac{y_4^{(4)}y_3^{(2)}y_3^{(1)}}{y_6^{(2)}} \newline
+ \frac{y_4^{(4)}y_3^{(2)}y_4^{(1)}}{y_4^{(3)}} +				\frac{y_4^{(4)}y_4^{(2)}y_4^{(1)}}{y_5^{(2)}} + \frac{y_4^{(4)}y_2^{(2)}y_4^{(3)}y_3^{(1)}}{y_3^{(3)}y_6^{(2)}}+ \frac{y_4^{(4)}y_2^{(2)}y_6^{(1)}y_2^{(1)}}{y_3^{(3)}y_3^{(2)}} + 		\frac{y_4^{(4)}y_2^{(2)}y_6^{(1)}y_3^{(1)}}	{y_3^{(3)}y_4^{(2)}} + \frac{y_4^{(4)}y_3^{(2)}y_6^{(1)}y_3^{(1)}}{y_4^{(3)}y_4^{(2)}}\newline
+\frac{y_4^{(4)}y_2^{(2)}y_4^{(3)}y_4^{(2)}		y_2^{(1)}}{y_3^{(3)}y_6^{(2)}y_3^{(2)}}\big) (-,+,-,+,-,-)+
y_6^{(2)}y_6^{(1)}(+,+,+,+,+,+) + \big(y_4^{(4)}y_4^{(1)} + \frac{y_4^{(4)}y_6^{(1)}y_0^{(1)}}{y_2^{(2)}} \newline
+ \frac{y_4^{(4)}y_4^{(3)}y_3^{(1)}}{y_6^{(2)}} + \frac{y_4^{(4)}y_6^{(1)}			y_2^{(1)}}{y_3^{(2)}} + \frac{y_4^{(4)}y_6^{(1)}y_3^{(1)}}{y_4^{(2)}} + \frac{y_4^{(4)}y_4^{(3)}y_4^{(2)}y_0^{(1)}}{y_2^{(2)}y_6^{(2)}}+ \frac{y_4^{(4)}y_4^{(3)}y_4^{(2)}x_2^{(1)}}		{y_6^{(2)}y_3^{(2)}} \big)(-,-,+,+,-,-) \newline
 + \big(y_5^{(3)}y_5^{(1)} + \frac{y_5^{(3)}y_2^{(2)}y_3^{(1)}}{y_4^{(4)}} + \frac{y_5^{(3)}y_2^{(2)}y_4^{(1)}}{y_3^{(3)}} + \frac{y_5^{(3)}			y_3^{(2)}y_3^{(1)}}{y_6^{(2)}} +	\frac{y_5^{(3)}y_3^{(2)}y_4^{(1)}}{y_4^{(3)}} + \frac{y_5^{(3)}y_4^{(2)}y_4^{(1)}}{y_5^{(2)}} + \frac{y_5^{(3)}y_2^{(2)}y_5^{(2)}y_2^{(1)}}{y_4^{(4)}	y_4^{(3)}} \newline
 + \frac{y_5^{(3)}y_2^{(2)}y_4^{(2)}y_2^{(1)}}{y_4^{(4)}y_3^{(2)}} + \frac{y_5^{(3)}y_2^{(2)}y_4^{(3)}y_3^{(1)}}{y_3^{(3)}y_6^{(2)}} + \frac{y_5^{(3)}y_2^{(2)}y_6^{(1)}		y_2^{(1)}}{y_3^{(3)}y_3^{(2)}} + \frac{y_5^{(3)}y_2^{(2)}y_6^{(1)}y_3^{(1)}}{y_3^{(3)}y_4^{(2)}} + \frac{y_5^{(3)}y_3^{(2)}y_6^{(1)}y_3^{(1)}}{y_4^{(3)}y_4^{(2)}} \newline + \frac{y_5^{(3)}		y_2^{(2)}y_4^{(3)}y_4^{(2)}y_2^{(1)}}{y_3^{(3)}y_6^{(2)}y_3^{(2)}}\big)  (-,+,-,-,+,-)+\big(y_4^{(4)}y_6^{(1)}  + \frac{y_4^{(4)}y_4^{(3)}y_4^{(2)}}{y_6^{(2)}}\big)(+,+,+,+,-,-) \newline
 + \big(y_5^{(3)}y_4^{(1)} + \frac{y_5^{(3)}y_3^{(3)}y_3^{(1)}}{y_4^{(4)}} + \frac{y_5^{(3)}y_6^{(1)}	y_0^{(1)}}{y_2^{(2)}} + \frac{y_5^{(3)}	y_4^{(3)}y_3^{(1)}}{y_6^{(2)}} +	\frac{y_5^{(3)}y_6^{(1)}y_2^{(1)}}{y_3^{(2)}} + \frac{y_5^{(3)}y_6^{(1)}y_3^{(1)}}{y_4^{(2)}} + \frac{y_5^{(3)}	y_3^{(3)}y_4^{(2)}y_0^{(1)}}{y_4^{(4)}y_2^{(2)}} \newline
 + \frac{y_5^{(3)}y_3^{(3)}y_5^{(2)}y_2^{(1)}}{y_4^{(4)}y_4^{(3)}} + \frac{y_5^{(3)}y_3^{(3)}y_4^{(2)}y_2^{(1)}}{y_4^{(4)}y_3^{(2)}} 	+ \frac{y_5^{(3)}y_4^{(3)}y_4^{(2)}y_0^{(1)}}{y_2^{(2)}y_6^{(2)}} + \frac{y_5^{(3)}y_4^{(3)}y_4^{(2)}y_2^{(1)}}{y_6^{(2)}y_3^{(2)}} + \frac{y_5^{(3)}y_3^{(3)}y_3^{(2)}y_5^{(2)}		y_0^{(1)}}{y_4^{(4)}y_2^{(2)}y_4^{(3)}}\big) \newline
 \times (-,-,+,-,+,-) + \big(y_5^{(1)} + \frac{y_2^{(2)}y_2^{(1)}}{y_5^{(3)}} + \frac{y_2^{(2)}y_3^{(1)}}{y_4^{(4)}} + \frac{y_2^{(2)}y_4^{(1)}}		{y_3^{(3)}} + \frac{y_3^{(2)}y_3^{(1)}}{y_6^{(2)}} + \frac{y_3^{(2)}y_4^{(1)}}{y_4^{(3)}} + \frac{y_4^{(2)}y_4^{(1)}}{y_5^{(2)}} \newline
 + \frac{y_2^{(2)}y_5^{(2)}y_2^{(1)}}{y_4^{(4)}			y_4^{(3)}} + \frac{y_2^{(2)}y_4^{(2)}y_2^{(1)}}{y_4^{(4)}y_3^{(2)}} + \frac{y_2^{(2)}y_4^{(3)}y_3^{(1)}}{y_3^{(3)}y_6^{(2)}} + \frac{y_2^{(2)}y_6^{(1)}y_2^{(1)}}{y_3^{(3)}y_3^{(2)}} 	+ \frac{y_2^{(2)}y_6^{(1)}y_3^{(1)}}{y_3^{(3)}y_4^{(2)}} + \frac{y_3^{(2)}y_6^{(1)}y_3^{(1)}}{y_4^{(3)}y_4^{(2)}} \newline
 +\frac{y_2^{(2)}y_4^{(3)}y_4^{(2)}y_2^{(1)}}{y_3^{(3)}y_6^{(2)}		y_3^{(2)}}\big) (-,+,-,-,-,+)+\big(y_5^{(3)}y_6^{(1)}  + \frac{y_5^{(3)}y_3^{(3)}y_4^{(2)}}{y_4^{(4)}} + \frac{y_5^{(3)}y_4^{(3)}y_4^{(2)}}{y_6^{(2)}} + \frac{y_5^{(3)}y_3^{(3)}y_3^{(2)}y_5^{(2)}}{y_4^{(4)}y_4^{(3)}}		\big) \newline
 \times (+,+,+,-,+,-) + \big(y_5^{(3)}y_3^{(1)}  + \frac{y_5^{(3)}y_5^{(2)}y_0^{(1)}}{y_3^{(3)}} + \frac{y_5^{(3)}y_4^{(2)}y_0^{(1)}}{y_2^{(2)}} 
 + \frac{y_5^{(3)}y_5^{(2)}				y_2^{(1)}}{y_4^{(3)}} + \frac{y_5^{(3)}y_4^{(2)}y_2^{(1)}}{y_3^{(2)}} \newline
 + \frac{y_5^{(3)}y_3^{(2)}y_5^{(2)}y_0^{(1)}}{y_2^{(2)}y_4^{(3)}}	\big)(-,-,-,+,+,-) + \big(y_4^{(1)} + 			\frac{y_3^{(3)}y_2^{(1)}}{y_5^{(3)}} + \frac{y_3^{(3)}y_3^{(1)}}{y_4^{(4)}} + \frac{y_6^{(1)}y_0^{(1)}}{y_2^{(2)}} + \frac{y_4^{(3)}y_3^{(1)}}{y_6^{(2)}} \newline
 + \frac{y_6^{(1)}y_2^{(1)}}		{y_3^{(2)}} + \frac{y_6^{(1)}y_3^{(1)}}{y_4^{(2)}} + \frac{y_3^{(3)}y_3^{(2)}y_0^{(1)}}{y_5^{(3)}y_2^{(2)}} + \frac{y_3^{(3)}y_4^{(2)}y_0^{(1)}}{y_4^{(4)}y_2^{(2)}} + \frac{y_3^{(3)}		y_5^{(2)}y_2^{(1)}}{y_4^{(4)}y_4^{(3)}} + \frac{y_3^{(3)}y_4^{(2)}y_2^{(1)}}{y_4^{(4)}y_3^{(2)}} + \frac{y_4^{(3)}y_4^{(2)}y_0^{(1)}}{y_2^{(2)}y_6^{(2)}} \newline
 + \frac{y_4^{(3)}y_4^{(2)}		y_2^{(1)}}	{y_6^{(2)}y_3^{(2)}} +\frac{y_3^{(3)}y_3^{(2)}y_5^{(2)}y_0^{(1)}}{y_4^{(4)}y_2^{(2)}y_4^{(3)}}\big)(-,-,+,-,-,+)+
 \big(y_5^{(3)}y_4^{(2)}  + \frac{y_5^{(3)}y_2^{(2)}y_5^{(2)}}{y_3^{(3)}} + \frac{y_5^{(3)}y_3^{(2)}y_5^{(2)}}{y_4^{(3)}} \big)\newline
 \times (+,+,-,+,+,-) + \big(y_6^{(1)} + \frac{y_3^{(3)}y_3^{(2)}}		{y_5^{(3)}} + \frac{y_3^{(3)}y_4^{(2)}}{y_4^{(4)}} + \frac{y_4^{(3)}y_4^{(2)}}	{y_6^{(2)}} + \frac{y_3^{(3)}y_3^{(2)}y_5^{(2)}}{y_4^{(4)}y_4^{(3)}} \big)(+,+,+,-,-,+) \newline
 + \big(y_3^{(1)} + 	\frac{y_4^{(4)}y_2^{(1)}}{y_5^{(3)}}   + \frac{y_5^{(2)}y_0^{(1)}}{y_3^{(3)}} + \frac{y_4^{(2)}y_0^{(1)}}{y_2^{(2)}} + \frac{y_5^{(2)}y_2^{(1)}}{y_4^{(3)}} + \frac{y_4^{(2)}y_2^{(1)}}		{y_3^{(2)}} + \frac{y_4^{(4)}y_4^{(3)}y_0^{(1)}}{y_5^{(3)}y_3^{(3)}}  + \frac{y_4^{(4)}y_3^{(2)}y_0^{(1)}}{y_5^{(3)}y_2^{(2)}} \newline + \frac{y_3^{(2)}	y_5^{(2)}y_0^{(1)}}{y_2^{(2)}y_4^{(3)}} 	\big)(-,-,-,+,-,+)+
y_5^{(3)}y_5^{(2)}(+,-,+,+,+,-) + \big(y_4^{(2)} + \frac{y_4^{(4)}y_3^{(2)}}{y_5^{(3)}} + \frac{y_2^{(2)}y_5^{(2)}}{y_3^{(3)}} \newline
+ \frac{y_3^{(2)}y_5^{(2)}}{y_4^{(3)}} + \frac{y_4^{(4)}			y_2^{(2)}y_4^{(3)}}{y_5^{(3)}y_3^{(3)}} \big)(+,+,-,+,-,+) + \big(y_2^{(1)} + \frac{y_6^{(2)}y_0^{(1)}}{y_4^{(4)}} + \frac{y_4^{(3)}y_0^{(1)}}{y_3^{(3)}} + \frac{y_3^{(2)}y_0^{(1)}}		{y_2^{(2)}} \big)\newline
\times(-,-,-,-,+,+)+\big(y_5^{(2)} + \frac{y_4^{(4)}y_4^{(3)}}{y_5^{(3)}} \big)(+,-,+,+,-,+) + \big(y_3^{(2)} + \frac{y_2^{(2)}y_6^{(2)}}{y_4^{(4)}} + \frac{y_2^{(2)}y_4^{(3)}}{y_3^{(3)}} \big)\newline
\times (+,+,-,-,+,+) + 			y_0^{(1)}(-,-,-,-,-,-)+
\big(y_4^{(3)} + \frac{y_3^{(3)}y_6^{(2)}}{y_4^{(4)}} \big)(+,-,+,-,+,+)\newline
 + y_2^{(2)}(+,+,-,-,-,-)+
y_6^{(2)}(+,-,-,+,+,+)+y_3^{(3)}(+,-,+,-,-,-) \newline
+y_4^{(4)}(+,-,-,+,-,-)+
y_5^{(3)}(+,-,-,-,+,-)+
(+,-,-,-,-,+)$.\\

Now for a given $x$ we solve the equation
\begin{align}
V_2(y) = a(x)\sigma (V_1(x)) \label{ytox}.
\end{align}
where $a(x)$ is a rational function in $x$ and the action of $\sigma$ on $V_1(x)$ is induced by its action on $W(\varpi_6)$. Though this equation is over-determined, it can be solved uniquely by comparing the coefficients of the basis vectors of $W(\varpi_6)$. We give the explicit  solutions of $a(x)$, and the variables $y_m^{(l)}$ below.
\begin{lemma} \label{yinx}
The rational function $a(x)$ and the complete solution of (\ref{ytox}) is:
\begin{align*} 
a(x) &=\frac{1}{x_5^{(2)}x_5^{(1)}}, \\
y_0^{(1)} 	&=\frac{x_6^{(3)}x_6^{(2)}x_6^{(1)}}{x_5^{(2)}x_5^{(1)}} , \qquad \qquad y_6^{(1)} =\frac{1}{x_5^{(2)}} , \qquad \qquad y_6^{(2)} =\frac{1}{x_5^{(1)}},\\
y_2^{(1)} 	&=\Big( \frac{x_5^{(2)}x_5^{(1)}}{x_6^{(3)}x_6^{(2)}}+\frac{x_6^{(1)}x_5^{(2)}x_5^{(1)}}{x_6^{(3)}x_4^{(2)}x_4^{(1)}}+\frac{x_6^{(2)}x_6^{(1)}  x_5^{(2)}x_5^{(1)}}{x_4^{(4)}x_4^{(3)}x_4^{(2)}x_4^{(1)}}\Big)^{-1}, \\
y_2^{(2)} 	&=\frac{x_4^{(4)}x_4^{(3)}x_4^{(2)}x_4^{(1)}}{(x_5^{(2)})^2(x_5^{(1)})^2}\Big( \frac{x_5^{(2)}x_5^{(1)}}{x_6^{(3)}x_6^{(2)}}+\frac{x_6^{(1)}x_5^{(2)}x_5^{(1)}}{x_6^{(3)}x_4^{(2)}x_4^{(1)}}+\frac{x_6^{(2)}x_6^{(1)}  x_5^{(2)}x_5^{(1)}}{x_4^{(4)}x_4^{(3)}x_4^{(2)}x_4^{(1)}}\Big), \\
y_3^{(1)} 	&=\Big( \frac{x_5^{(2)}x_5^{(1)}}{x_6^{(3)}x_4^{(2)}}+\frac{x_6^{(2)}x_5^{(2)}x_5^{(1)}}{x_4^{(4)}x_4^{(3)}x_4^{(2)}}+\frac{x_5^{(2)}x_4^{(1)}}{x_6^{(3)}x_3^{(1)}}+\frac{x_6^{(2)}x_5^{(2)}x_4^{(1)}}{x_4^{(4)}x_4^{(3)}x_3^{(1)}}+\frac{x_5^{(2)}x_4^{(2)}x_4^{(1)}}{x_4^{(4)}x_3^{(2)}x_3^{(1)}} \\
		& \qquad +\frac{x_4^{(3)}x_4^{(2)}x_4^{(1)}}{x_3^{(3)}x_3^{(2)}x_3^{(1)}}\Big)^{-1},\\
y_3^{(2)} 	&= \frac{x_3^{(3)}x_3^{(2)}x_3^{(1)}}{(x_5^{(2)})^2(x_5^{(1)})^2} \Big( \frac{x_5^{(2)}x_5^{(1)}}{x_6^{(3)}x_4^{(2)}}+\frac{x_6^{(2)}x_5^{(2)}x_5^{(1)}}{x_4^{(4)}x_4^{(3)}x_4^{(2)}}+\frac{x_5^{(2)}x_4^{(1)}}{x_6^{(3)}x_3^{(1)}}+\frac{x_6^{(2)}x_5^{(2)}x_4^{(1)}}{x_4^{(4)}x_4^{(3)}x_3^{(1)}} \\	
		& \qquad +\frac{x_5^{(2)}x_4^{(2)}x_4^{(1)}}{x_4^{(4)}x_3^{(2)}x_3^{(1)}}+\frac{x_4^{(3)}x_4^{(2)}x_4^{(1)}}{x_3^{(3)}x_3^{(2)}x_3^{(1)}}\Big)\Big(\frac{x_3^{(3)}x_3^{(2)}x_3^{(1)}}{x_4^{(4)}x_4^{(3)}x_4^{(2)}}+\frac{x_4^{(1)}x_3^{(3)}x_3^{(2)}}{x_5^{(1)}x_4^{(4)}x_4^{(3)}}+\frac{x_4^{(2)}x_4^{(1)}x_3^{(3)}}{x_6^{(2)}x_5^{(1)}x_4^{(4)}} \\
		& \qquad +\frac{x_6^{(1)}x_3^{(3)}}{x_5^{(1)}x_4^{(4)}}+\frac{x_4^{(3)}x_4^{(2)}x_4^{(1)}}{x_6^{(2)}x_5^{(2)}x_5^{(1)}}+\frac{x_6^{(1)}x_4^{(3)}}{x_5^{(2)}x_5^{(1)}} \Big)^{-1},\\
y_3^{(3)} 	&= \frac{x_3^{(3)}x_3^{(2)}x_3^{(1)}}{x_4^{(4)}x_4^{(3)}x_4^{(2)}}+\frac{x_4^{(1)}x_3^{(3)}x_3^{(2)}}{x_5^{(1)}x_4^{(4)}x_4^{(3)}}+\frac{x_4^{(2)}x_4^{(1)}x_3^{(3)}}{x_6^{(2)}x_5^{(1)}x_4^{(4)}}+\frac{x_6^{(1)}x_3^{(3)}}{x_5^{(1)}x_4^{(4)}}+\frac{x_4^{(3)}x_4^{(2)}x_4^{(1)}}{x_6^{(2)}x_5^{(2)}x_5^{(1)}} \\
		& \qquad +\frac{x_6^{(1)}x_4^{(3)}}{x_5^{(2)}x_5^{(1)}}, \\
y_4^{(1)} 	&= \Big( \frac{x_5^{(2)}}{x_6^{(3)}}+\frac{x_6^{(2)}x_5^{(2)}}{x_4^{(4)}x_4^{(3)}}+\frac{x_5^{(2)}x_4^{(2)}}{x_4^{(4)}x_3^{(2)}}+\frac{x_4^{(3)}x_4^{(2)}}{x_3^{(3)}x_3^{(2)}}+\frac{x_5^{(2)}x_3^{(1)}}{x_4^{(4)}x_2^{(1)}}+\frac{x_4^{(3)}x_3^{(1)}}{x_3^{(3)}x_2^{(1)}}+\frac{x_3^{(2)}x_3^{(1)}}{x_2^{(2)}x_2^{(1)}} \Big)^{-1}, \\
y_4^{(2)} 	&=  \frac{x_2^{(2)}x_2^{(1)}}{(x_5^{(2)})^2(x_5^{(1)})^2} \Big( \frac{x_5^{(2)}}{x_6^{(3)}}+\frac{x_6^{(2)}x_5^{(2)}}{x_4^{(4)}x_4^{(3)}}+\frac{x_5^{(2)}x_4^{(2)}}{x_4^{(4)}x_3^{(2)}}+\frac{x_4^{(3)}x_4^{(2)}}{x_3^{(3)}x_3^{(2)}}+\frac{x_5^{(2)}x_3^{(1)}}{x_4^{(4)}x_2^{(1)}}+\frac{x_4^{(3)}x_3^{(1)}}{x_3^{(3)}x_2^{(1)}} \\ 
		& \qquad  +\frac{x_3^{(2)}x_3^{(1)}}{x_2^{(2)}x_2^{(1)}} \Big)\Big( \frac{x_2^{(2)}x_2^{(1)}}{x_5^{(1)}x_4^{(4)}x_3^{(2)}}+\frac{x_4^{(3)}x_2^{(2)}x_2^{(1)}}{x_5^{(2)}x_5^{(1)}x_3^{(3)}x_3^{(2)}}+\frac{x_3^{(1)}x_2^{(2)}}{x_5^{(1)}x_4^{(4)}x_4^{(2)}}\\
		& \qquad +\frac{x_4^{(3)}x_3^{(1)}x_2^{(2)}}{x_5^{(2)}x_5^{(1)}x_4^{(2)}x_3^{(3)}}+\frac{x_4^{(1)}x_2^{(2)}}{(x_5^{(1)})^2 x_4^{(4)}} + \frac{x_4^{(3)}x_4^{(1)}x_2^{(2)}}{x_5^{(2)} (x_5^{(1)})^2 x_3^{(3)}}+\frac{x_3^{(2)}x_3^{(1)}}{x_5^{(2)}x_5^{(1)}x_4^{(2)}} \\
		& \qquad +\frac{x_4^{(1)}x_3^{(2)}}{x_5^{(2)} (x_5^{(1)})^2} \Big)^{-1},\\
y_4^{(3)} 	&= \Big( \frac{x_2^{(2)}x_2^{(1)}}{x_5^{(1)}x_4^{(4)}x_3^{(2)}}+\frac{x_4^{(3)}x_2^{(2)}x_2^{(1)}}{x_5^{(2)}x_5^{(1)}x_3^{(3)}x_3^{(2)}}+\frac{x_3^{(1)}x_2^{(2)}}{x_5^{(1)}x_4^{(4)}x_4^{(2)}}+\frac{x_4^{(3)}x_3^{(1)}x_2^{(2)}}{x_5^{(2)}x_5^{(1)}x_4^{(2)}x_3^{(3)}}\\
		& \qquad +\frac{x_4^{(1)}x_2^{(2)}}{(x_5^{(1)})^2 x_4^{(4)}} +\frac{x_4^{(3)}x_4^{(1)}x_2^{(2)}}{x_5^{(2)} (x_5^{(1)})^2 x_3^{(3)}}+\frac{x_3^{(2)}x_3^{(1)}}{x_5^{(2)}x_5^{(1)}x_4^{(2)}}+\frac{x_4^{(1)}x_3^{(2)}}{x_5^{(2)} (x_5^{(1)})^2} \Big) \Big(\frac{x_2^{(2)}x_2^{(1)}}{x_3^{(3)}x_3^{(2)}}\\
		& \qquad +\frac{x_3^{(1)}x_2^{(2)}}{x_4^{(2)}x_3^{(3)}}+\frac{x_4^{(1)}x_2^{(2)}}{x_5^{(1)}x_3^{(3)}}+ \frac{x_3^{(2)}x_3^{(1)}}{x_4^{(3)}x_4^{(2)}}+\frac{x_4^{(1)}x_3^{(2)}}{x_5^{(1)}x_4^{(3)}}+\frac{x_4^{(2)}x_4^{(1)}}{x_6^{(2)}x_5^{(1)}}+\frac{x_6^{(1)}}{x_5^{(1)}}\Big)^{-1}, \\
y_4^{(4)} 	&= \frac{x_2^{(2)}x_2^{(1)}}{x_3^{(3)}x_3^{(2)}}+\frac{x_3^{(1)}x_2^{(2)}}{x_4^{(2)}x_3^{(3)}}+\frac{x_4^{(1)}x_2^{(2)}}{x_5^{(1)}x_3^{(3)}}+\frac{x_3^{(2)}x_3^{(1)}}{x_4^{(3)}x_4^{(2)}}+\frac{x_4^{(1)}x_3^{(2)}}{x_5^{(1)}x_4^{(3)}}+\frac{x_4^{(2)}x_4^{(1)}}{x_6^{(2)}x_5^{(1)}}+\frac{x_6^{(1)}}{x_5^{(1)}}, \\
y_5^{(1)} 	&= \Big( \frac{x_5^{(2)}}{x_4^{(4)}}+\frac{x_4^{(3)}}{x_3^{(3)}}+\frac{x_3^{(2)}}{x_2^{(2)}}+\frac{x_2^{(1)}}{x_1^{(1)}} \Big)^{-1}, \\
y_5^{(2)} 	&=  \frac{x_1^{(1)}}{x_5^{(2)}x_5^{(1)}}\Big( \frac{x_5^{(2)}}{x_4^{(4)}}+\frac{x_4^{(3)}}{x_3^{(3)}}+\frac{x_3^{(2)}}{x_2^{(2)}}+\frac{x_2^{(1)}}{x_1^{(1)}} \Big) \Big( \frac{x_1^{(1)}}{x_2^{(2)}}+\frac{x_2^{(1)}}{x_3^{(2)}}+\frac{x_3^{(1)}}{x_4^{(2)}}+\frac{x_4^{(1)}}{x_5^{(1)}} \Big)^{-1}, \\
y_5^{(3)} 	&=  \frac{x_1^{(1)}}{x_2^{(2)}}+\frac{x_2^{(1)}}{x_3^{(2)}}+\frac{x_3^{(1)}}{x_4^{(2)}}+\frac{x_4^{(1)}}{x_5^{(1)}}. 
\end{align*}
\end{lemma}

Using Lemma \ref{yinx} we define the map
$$\bar{\sigma}: \mathcal{V}_1 \rightarrow \mathcal{V}_2,$$
$$V_1(x) \mapsto V_2(y).$$
Now we have the following result.

\begin{prop}\label{mapsigmabar} 
The map $\bar{\sigma}: \mathcal{V}_1 \rightarrow \mathcal{V}_2$ is a bi-positive birational isomorphism with the inverse positive rational map 
$$\bar{\sigma}^{-1}: \mathcal{V}_2 \rightarrow \mathcal{V}_1,$$
$$V_2(y) \mapsto V_1(x)$$
given by
\begin{align*} 
x_1^{(1)} 	&=\frac{y_5^{(3)}y_5^{(2)}y_5^{(1)}}{y_6^{(2)}y_6^{(1)}} , \qquad \qquad x_5^{(1)} =\frac{1}{y_6^{(2)}} , \qquad \qquad x_5^{(2)} =\frac{1}{y_6^{(1)}},\\
x_2^{(1)} 	&=\Big( \frac{y_6^{(2)}y_6^{(1)}}{y_5^{(3)}y_5^{(2)}}+\frac{y_6^{(2)}y_6^{(1)}y_5^{(1)}}{y_5^{(3)}y_4^{(2)}y_4^{(1)}}+\frac{y_6^{(2)}y_6^{(1)}  y_5^{(2)}y_5^{(1)}}{y_4^{(4)}y_4^{(3)}y_4^{(2)}y_4^{(1)}}\Big)^{-1}, \\
x_2^{(2)} 	&=\frac{y_4^{(4)}y_4^{(3)}y_4^{(2)}y_4^{(1)}}{(y_6^{(2)})^2(y_6^{(1)})^2}\Big(  \frac{y_6^{(2)}y_6^{(1)}}{y_5^{(3)}y_5^{(2)}}+\frac{y_6^{(2)}y_6^{(1)}y_5^{(1)}}{y_5^{(3)}y_4^{(2)}y_4^{(1)}}+\frac{y_6^{(2)}y_6^{(1)}  y_5^{(2)}y_5^{(1)}}{y_4^{(4)}y_4^{(3)}y_4^{(2)}y_4^{(1)}}\Big), \\
x_3^{(1)} 	&=\Big( \frac{y_6^{(2)}y_6^{(1)}}{y_5^{(3)}y_4^{(2)}}+\frac{y_6^{(2)}y_6^{(1)}y_5^{(2)}}{y_4^{(4)}y_4^{(3)}y_4^{(2)}}+\frac{y_6^{(2)}y_4^{(1)}}{y_5^{(3)}y_3^{(1)}}+\frac{y_6^{(2)}y_5^{(2)}y_4^{(1)}}{y_4^{(4)}y_4^{(3)}y_3^{(1)}}+\frac{y_6^{(2)}y_4^{(2)}y_4^{(1)}}{y_4^{(4)}y_3^{(2)}y_3^{(1)}} \\
		& \qquad +\frac{y_4^{(3)}y_4^{(2)}y_4^{(1)}}{y_3^{(3)}y_3^{(2)}y_3^{(1)}}\Big)^{-1},\\
x_3^{(2)} 	&= \frac{y_3^{(3)}y_3^{(2)}y_3^{(1)}}{(y_6^{(2)})^2(y_6^{(1)})^2} \Big(  \frac{y_6^{(2)}y_6^{(1)}}{y_5^{(3)}y_4^{(2)}}+\frac{y_6^{(2)}y_6^{(1)}y_5^{(2)}}{y_4^{(4)}y_4^{(3)}y_4^{(2)}}+\frac{y_6^{(2)}y_4^{(1)}}{y_5^{(3)}y_3^{(1)}}+\frac{y_6^{(2)}y_5^{(2)}y_4^{(1)}}{y_4^{(4)}y_4^{(3)}y_3^{(1)}}\\	
		& \qquad +\frac{y_6^{(2)}y_4^{(2)}y_4^{(1)}}{y_4^{(4)}y_3^{(2)}y_3^{(1)}}+\frac{y_4^{(3)}y_4^{(2)}y_4^{(1)}}{y_3^{(3)}y_3^{(2)}y_3^{(1)}} \Big)  \Big( \frac{y_3^{(3)}y_3^{(2)}y_3^{(1)}}{y_4^{(4)}y_4^{(3)}y_4^{(2)}}+\frac{y_4^{(1)}y_3^{(3)}y_3^{(2)}}{y_6^{(1)}y_4^{(4)}y_4^{(3)}}+\frac{y_4^{(2)}y_4^{(1)}y_3^{(3)}}{y_6^{(1)}y_5^{(2)}y_4^{(4)}}\\
		& \qquad+\frac{y_5^{(1)}y_3^{(3)}}{y_6^{(1)}y_4^{(4)}}+\frac{y_4^{(3)}y_4^{(2)}y_4^{(1)}}{y_6^{(2)}y_6^{(1)}y_5^{(2)}}+\frac{y_5^{(1)}y_4^{(3)}}{y_6^{(2)}y_6^{(1)}} \Big)^{-1}, \\
x_3^{(3)} 	&=  \frac{y_3^{(3)}y_3^{(2)}y_3^{(1)}}{y_4^{(4)}y_4^{(3)}y_4^{(2)}}+\frac{y_4^{(1)}y_3^{(3)}y_3^{(2)}}{y_6^{(1)}y_4^{(4)}y_4^{(3)}}+\frac{y_4^{(2)}y_4^{(1)}y_3^{(3)}}{y_6^{(1)}y_5^{(2)}y_4^{(4)}}+\frac{y_5^{(1)}y_3^{(3)}}{y_6^{(1)}y_4^{(4)}}+\frac{y_4^{(3)}y_4^{(2)}y_4^{(1)}}{y_6^{(2)}y_6^{(1)}y_5^{(2)}}+\frac{y_5^{(1)}y_4^{(3)}}{y_6^{(2)}y_6^{(1)}}, \\
x_4^{(1)} 	&= \Big( \frac{y_6^{(2)}}{y_5^{(3)}}+\frac{y_6^{(2)}y_5^{(2)}}{y_4^{(4)}y_4^{(3)}}+\frac{y_6^{(2)}y_4^{(2)}}{y_4^{(4)}y_3^{(2)}}+\frac{y_4^{(3)}y_4^{(2)}}{y_3^{(3)}y_3^{(2)}}+\frac{y_6^{(2)}y_3^{(1)}}{y_4^{(4)}y_2^{(1)}}+\frac{y_4^{(3)}y_3^{(1)}}{y_3^{(3)}y_2^{(1)}}+\frac{y_3^{(2)}y_3^{(1)}}{y_2^{(2)}y_2^{(1)}} \Big)^{-1}, \\
x_4^{(2)} 	&=  \frac{y_2^{(2)}y_2^{(1)}}{(y_6^{(2)})^2(y_6^{(1)})^2} \Big(  \frac{y_6^{(2)}}{y_5^{(3)}}+\frac{y_6^{(2)}y_5^{(2)}}{y_4^{(4)}y_4^{(3)}}+\frac{y_6^{(2)}y_4^{(2)}}{y_4^{(4)}y_3^{(2)}}+\frac{y_4^{(3)}y_4^{(2)}}{y_3^{(3)}y_3^{(2)}}+\frac{y_6^{(2)}y_3^{(1)}}{y_4^{(4)}y_2^{(1)}}+\frac{y_4^{(3)}y_3^{(1)}}{y_3^{(3)}y_2^{(1)}} \\ 
		& \qquad +\frac{y_3^{(2)}y_3^{(1)}}{y_2^{(2)}y_2^{(1)}} \Big) \Big( \frac{y_2^{(2)}y_2^{(1)}}{y_6^{(1)}y_4^{(4)}y_3^{(2)}}+\frac{y_4^{(3)}y_2^{(2)}y_2^{(1)}}{y_6^{(2)}y_6^{(1)}y_3^{(3)}y_3^{(2)}}+\frac{y_3^{(1)}y_2^{(2)}}{y_6^{(1)}y_4^{(4)}y_4^{(2)}}\\
		& \qquad +\frac{y_4^{(3)}y_3^{(1)}y_2^{(2)}}{y_6^{(2)}y_6^{(1)}y_4^{(2)}y_3^{(3)}}+\frac{y_4^{(1)}y_2^{(2)}}{(y_6^{(1)})^2 y_4^{(4)}} +\frac{y_4^{(3)}y_4^{(1)}y_2^{(2)}}{y_6^{(2)} (y_6^{(1)})^2 y_3^{(3)}}+\frac{y_3^{(2)}y_3^{(1)}}{y_6^{(2)}y_6^{(1)}y_4^{(2)}} \\
		& \qquad +\frac{y_4^{(1)}y_3^{(2)}}{y_6^{(2)} (y_6^{(1)})^2} \Big)^{-1},\\
x_4^{(3)} 	&= \Big( \frac{y_2^{(2)}y_2^{(1)}}{y_6^{(1)}y_4^{(4)}y_3^{(2)}}+\frac{y_4^{(3)}y_2^{(2)}y_2^{(1)}}{y_6^{(2)}y_6^{(1)}y_3^{(3)}y_3^{(2)}}+\frac{y_3^{(1)}y_2^{(2)}}{y_6^{(1)}y_4^{(4)}y_4^{(2)}}+\frac{y_4^{(3)}y_3^{(1)}y_2^{(2)}}{y_6^{(2)}y_6^{(1)}y_4^{(2)}y_3^{(3)}}\\
		& \qquad +\frac{y_4^{(1)}y_2^{(2)}}{(y_6^{(1)})^2 y_4^{(4)}} +\frac{y_4^{(3)}y_4^{(1)}y_2^{(2)}}{y_6^{(2)} (y_6^{(1)})^2 y_3^{(3)}}+\frac{y_3^{(2)}y_3^{(1)}}{y_6^{(2)}y_6^{(1)}y_4^{(2)}}+\frac{y_4^{(1)}y_3^{(2)}}{y_6^{(2)} (y_6^{(1)})^2} \Big) \Big(\frac{y_2^{(2)}y_2^{(1)}}{y_3^{(3)}y_3^{(2)}}\\
		& \qquad +\frac{y_3^{(1)}y_2^{(2)}}{y_4^{(2)}y_3^{(3)}}+\frac{y_4^{(1)}y_2^{(2)}}{y_6^{(1)}y_3^{(3)}}+\frac{y_3^{(2)}y_3^{(1)}}{y_4^{(3)}y_4^{(2)}}+\frac{y_4^{(1)}y_3^{(2)}}{y_6^{(1)}y_4^{(3)}}+\frac{y_4^{(2)}y_4^{(1)}}{y_6^{(1)}y_5^{(2)}}+\frac{y_5^{(1)}}{y_6^{(1)}}\Big)^{-1}, \\
x_4^{(4)} 	&= \frac{y_2^{(2)}y_2^{(1)}}{y_3^{(3)}y_3^{(2)}}+\frac{y_3^{(1)}y_2^{(2)}}{y_4^{(2)}y_3^{(3)}}+\frac{y_4^{(1)}y_2^{(2)}}{y_6^{(1)}y_3^{(3)}}+\frac{y_3^{(2)}y_3^{(1)}}{y_4^{(3)}y_4^{(2)}}+\frac{y_4^{(1)}y_3^{(2)}}{y_6^{(1)}y_4^{(3)}}+\frac{y_4^{(2)}y_4^{(1)}}{y_6^{(1)}y_5^{(2)}}+\frac{y_5^{(1)}}{y_6^{(1)}}, \\
x_6^{(1)} 	&= \Big( \frac{y_6^{(2)}}{y_4^{(4)}}+\frac{y_4^{(3)}}{y_3^{(3)}}+\frac{y_3^{(2)}}{y_2^{(2)}}+\frac{y_2^{(1)}}{y_0^{(1)}} \Big)^{-1}, \\
x_6^{(2)} 	&=  \frac{y_0^{(1)}}{y_6^{(2)}y_6^{(1)}}\Big( \frac{y_6^{(2)}}{y_4^{(4)}}+\frac{y_4^{(3)}}{y_3^{(3)}}+\frac{y_3^{(2)}}{y_2^{(2)}}+\frac{y_2^{(1)}}{y_0^{(1)}}  \Big) \Big( \frac{y_0^{(1)}}{y_2^{(2)}}+\frac{y_2^{(1)}}{y_3^{(2)}}+\frac{y_3^{(1)}}{y_4^{(2)}}+\frac{y_4^{(1)}}{y_6^{(1)}}\Big)^{-1}, \\
x_6^{(3)} 	&=  \frac{y_0^{(1)}}{y_2^{(2)}}+\frac{y_2^{(1)}}{y_3^{(2)}}+\frac{y_3^{(1)}}{y_4^{(2)}}+\frac{y_4^{(1)}}{y_6^{(1)}}.
\end{align*}
\end{prop}
\begin{proof} The fact that $\bar{\sigma}$ is a bi-positive birational map follows from the explicit formulas. The rest follows by direct calculations.
\end{proof}
It is known that $\cV_1$ (resp. $\cV_2$) has the structure of a $\ge_0$ (resp. $\ge_1$) positive geometric crystal (\cite{BK}, \cite{N}, \cite{KNO}). Indeed, note that taking the sesquence  ${\bf i} = (6,4,3,2,5,4,3,6,4,5,1,2,3,4,6)$ the explicit actions of \ $e_k^c, \ \gamma_k, \ \veps_k$ on $V_1(x)$ for $k=1,2,3,4,5,6$ are given by Theorem \ref{schubert}  as follows.

\begin{align*}
e_k^c(V_1(x)) &=
\begin{cases}
V_1(x_6^{(3)},x_4^{(4)}, x_3^{(3)}, x_2^{(2)}, x_5^{(2)}, x_4^{(3)}, x_3^{(2)}, x_6^{(2)}, x_4^{(2)}, x_5^{(1)},cx_1^{(1)}, x_2^{(1)}, x_3^{(1)},\\
	\quad  x_4^{(1)}, x_6^{(1)}), \  k=1,\\
V_1(x_6^{(3)},x_4^{(4)}, x_3^{(3)}, c_2 x_2^{(2)}, x_5^{(2)}, x_4^{(3)}, x_3^{(2)}, x_6^{(2)}, x_4^{(2)}, x_5^{(1)},x_1^{(1)}, \frac{c}{c_2} x_2^{(1)}, \\
	\quad x_3^{(1)}, x_4^{(1)}, x_6^{(1)}), \ k=2,\\
V_1(x_6^{(3)},x_4^{(4)}, c_{3_1} x_3^{(3)}, x_2^{(2)}, x_5^{(2)}, x_4^{(3)}, c_{3_2} x_3^{(2)}, x_6^{(2)}, x_4^{(2)}, x_5^{(1)},x_1^{(1)}, x_2^{(1)}, \\
	\quad  \frac{c}{c_{3_1} c_{3_2}} x_3^{(1)}, x_4^{(1)}, x_6^{(1)}), \ k=3,\\
V_1(x_6^{(3)}, c_{4_1} x_4^{(4)}, x_3^{(3)}, x_2^{(2)}, x_5^{(2)}, c_{4_2} x_4^{(3)}, x_3^{(2)}, x_6^{(2)}, c_{4_3} x_4^{(2)}, x_5^{(1)},x_1^{(1)}, \\
	\quad x_2^{(1)}, x_3^{(1)}, \frac{c}{c_{4_1} c_{4_2} c_{4_3}} x_4^{(1)}, x_6^{(1)}), \ k=4,\\
V_1(x_6^{(3)},x_4^{(4)}, x_3^{(3)}, x_2^{(2)},c_5 x_5^{(2)}, x_4^{(3)}, x_3^{(2)}, x_6^{(2)}, x_4^{(2)},  \frac{c}{c_5} x_5^{(1)},x_1^{(1)}, x_2^{(1)}, \\
	\quad x_3^{(1)}, x_4^{(1)}, x_6^{(1)}), \ k=5,\\
V_1(c_{6_1}x_6^{(3)},x_4^{(4)}, x_3^{(3)}, x_2^{(2)}, x_5^{(2)}, x_4^{(3)}, x_3^{(2)}, c_{6_2} x_6^{(2)}, x_4^{(2)},   x_5^{(1)},x_1^{(1)}, x_2^{(1)}, \\
	\quad x_3^{(1)}, x_4^{(1)}, \frac{c}{c_{6_1} c_{6_2}} x_6^{(1)}), \ k=6,\\
\end{cases} \\
\text{where} & \\
&c_2 =\frac{cx_2^{(2)}x_2^{(1)}+x_3^{(2)}x_1^{(1)}}{x_2^{(2)}x_2^{(1)}+x_3^{(2)}x_1^{(1)}}, \\
&c_{3_1} = \frac{cx_3^{(3)}(x_3^{(2)})^2 x_3^{(1)} + x_4^{(3)}x_3^{(2)}x_3^{(1)}x_2^{(2)} + x_4^{(3)}x_4^{(2)}x_2^{(2)}x_2^{(1)}}{x_3^{(3)}(x_3^{(2)})^2 x_3^{(1)} + x_4^{(3)}x_3^{(2)}x_3^{(1)}x_2^{(2)} + x_4^{(3)}x_4^{(2)}x_2^{(2)}x_2^{(1)}}, \\
&c_{3_2} = \frac{cx_3^{(3)}(x_3^{(2)})^2 x_3^{(1)} + c x_4^{(3)}x_3^{(2)}x_3^{(1)}x_2^{(2)} + x_4^{(3)}x_4^{(2)}x_2^{(2)}x_2^{(1)}}{cx_3^{(3)}(x_3^{(2)})^2 x_3^{(1)} + x_4^{(3)}x_3^{(2)}x_3^{(1)}x_2^{(2)} + x_4^{(3)}x_4^{(2)}x_2^{(2)}x_2^{(1)}}, \\
&c_{4_1} = \frac{cc_4^1 + c_4^2 + c_4^3 + c_4^4}{c_4^1 + c_4^2 + c_4^3 + c_4^4}, \\
&c_{4_2} = \frac{cc_4^1 + cc_4^2 + c_4^3 + c_4^4}{cc_4^1 + c_4^2 + c_4^3 + c_4^4}, \\
&c_{4_3} = \frac{cc_4^1 + cc_4^2 + cc_4^3 + c_4^4}{cc_4^1 + cc_4^2 + c_4^3 + c_4^4}, \\
& c_4^1 = x_4^{(4)}(x_4^{(3)})^2(x_4^{(2)})^2 x_4^{(1)},\\
& c_4^2 = x_5^{(2)} x_4^{(3)}(x_4^{(2)})^2 x_4^{(1)}x_3^{(3)},\\
& c_4^3 = x_6^{(2)}x_5^{(2)}x_4^{(2)}x_4^{(1)}x_3^{(3)}x_3^{(2)},\\
& c_4^4 = x_6^{(2)}x_5^{(2)}x_5^{(1)}x_3^{(3)}x_3^{(2)}x_3^{(1)},\\
&c_5 = \frac{cx_5^{(2)}x_5^{(1)}+x_4^{(3)}x_4^{(2)}}{x_5^{(2)}x_5^{(1)}+x_4^{(3)}x_4^{(2)}}, \\
&c_{6_1} = \frac{cx_6^{(3)}(x_6^{(2)})^2 x_6^{(1)} + x_6^{(2)}x_6^{(1)}x_4^{(4)}x_4^{(3)} + x_4^{(4)}x_4^{(3)}x_4^{(2)}x_4^{(1)}}{x_6^{(3)}(x_6^{(2)})^2 x_6^{(1)} + x_6^{(2)}x_6^{(1)}x_4^{(4)}x_4^{(3)} + x_4^{(4)}x_4^{(3)}x_4^{(2)}x_4^{(1)}}, \\
&c_{6_2} = \frac{cx_6^{(3)}(x_6^{(2)})^2 x_6^{(1)} + c x_6^{(2)}x_6^{(1)}x_4^{(4)}x_4^{(3)} + x_4^{(4)}x_4^{(3)}x_4^{(2)}x_4^{(1)}}{c x_6^{(3)}(x_6^{(2)})^2 x_6^{(1)} + x_6^{(2)}x_6^{(1)}x_4^{(4)}x_4^{(3)} + x_4^{(4)}x_4^{(3)}x_4^{(2)}x_4^{(1)}}. \\
\gamma_k(V_1(x)) &= 
\begin{cases}
\frac{(x_1^{(1)})^2}{x_2^{(2)}x_2^{(1)}}, &k=1,\\
\frac{(x_2^{(2)})^2 (x_2^{(1)})^2}{x_3^{(3)}x_3^{(2)}x_3^{(1)}x_1^{(1)}}, &k=2, \\ 
\frac{(x_3^{(3)})^2 (x_3^{(2)})^2 (x_3^{(1)})^2}{x_4^{(4)}x_4^{(3)}x_4^{(2)}x_4^{(1)}x_2^{(2)}x_2^{(1)}}, &k=3, \\ 
\frac{(x_4^{(4)})^2 (x_4^{(3)})^2 (x_4^{(2)})^2 (x_4^{(1)})^2}{x_6^{(3)}x_6^{(2)}x_6^{(1)}x_5^{(2)}x_5^{(1)}x_3^{(3)}x_3^{(2)}x_3^{(1)}}, &k=4, \\ 
\frac{(x_5^{(2)})^2 (x_5^{(1)})^2}{x_4^{(4)}x_4^{(3)}x_4^{(2)}x_4^{(1)}}, &k=5, \\ 
\frac{(x_6^{(3)})^2 (x_6^{(2)})^2 (x_6^{(1)})^2}{x_4^{(4)}x_4^{(3)}x_4^{(2)}x_4^{(1)}}, &k=6. \\ 
\end{cases}\\
\veps_k(V_1(x)) &= 
\begin{cases}
\frac{x_2^{(2)}}{x_1^{(1)}}, &k=1, \\ 
\frac{x_3^{(3)}}{x_2^{(2)}} \big( 1 + \frac{x_3^{(2)}x_1^{(1)}}{x_2^{(2)}x_2^{(1)}} \big), &k=2, \\ 
\frac{x_4^{(4)}}{x_3^{(3)}} \big( 1 + \frac{x_4^{(3)}x_2^{(2)}}{x_3^{(3)}x_3^{(2)}}+ \frac{x_4^{(3)}x_4^{(2)}x_2^{(2)}x_2^{(1)}}{x_3^{(3)}(x_3^{(2)})^2 x_3^{(1)}} \big), &k=3, \\ 
\frac{x_6^{(3)}}{x_4^{(4)}} \big( 1 + \frac{x_5^{(2)}x_3^{(3)}}{x_4^{(4)}x_4^{(3)}}+ \frac{x_6^{(2)}x_5^{(2)}x_3^{(3)}x_3^{(2)}}{x_4^{(4)}(x_4^{(3)})^2 x_4^{(2)}} + \frac{x_6^{(2)}x_5^{(2)}x_5^{(1)}x_3^{(3)}x_3^{(2)}x_3^{(1)}}{x_4^{(4)}(x_4^{(3)})^2 (x_4^{(2)})^2 x_4^{(1)}} \big), &k=4, \\ 
\frac{x_4^{(4)}}{x_5^{(2)}} \big( 1 + \frac{x_4^{(3)}x_4^{(2)}}{x_5^{(2)}x_5^{(1)}} \big), &k=5, \\ 
\frac{1}{x_6^{(3)}} \big( 1 + \frac{x_4^{(4)}x_4^{(3)}}{x_6^{(3)}x_6^{(2)}}+ \frac{x_4^{(4)}x_4^{(3)}x_4^{(2)}x_4^{(1)}}{x_6^{(3)}(x_6^{(2)})^2 x_6^{(1)}} \big), &k=6. \\ 
\end{cases} 
\end{align*}

By choosing ${\bf i} = (5,4,3,2,6,4,3,5,4,6,0,2,3,4,5)$, we also have the following explicit actions of \ $\bar{e_k}^c,  \ \bar{\gamma}_k, \ \bar{\veps}_k, \  k=0,2,3,4,5,6$ on $V_2(y)$  by Theorem \ref{schubert}.
\begin{align*}
\bar{e_k}^c(V_2(y)) &=
\begin{cases}
V_2(y_5^{(3)},y_4^{(4)}, y_3^{(3)}, y_2^{(2)}, y_6^{(2)}, y_4^{(3)}, y_3^{(2)}, y_5^{(2)}, y_4^{(2)}, y_6^{(1)}, c y_0^{(1)},y_2^{(1)},y_3^{(1)},\\
	\quad y_4^{(1)},y_5^{(1)}),  \ k=0,\\
V_2(y_5^{(3)},y_4^{(4)}, y_3^{(3)},  \bar{c}_2 y_2^{(2)}, y_6^{(2)}, y_4^{(3)}, y_3^{(2)}, y_5^{(2)}, y_4^{(2)}, y_6^{(1)}, y_0^{(1)}, \frac{c}{\bar{c}_2}y_2^{(1)},\\
	\quad y_3^{(1)},y_4^{(1)},y_5^{(1)}), \ k=2,\\
V_2(y_5^{(3)},y_4^{(4)}, \bar{c}_{3_1} y_3^{(3)}, y_2^{(2)}, y_6^{(2)}, y_4^{(3)},  \bar{c}_{3_2} y_3^{(2)}, y_5^{(2)}, y_4^{(2)}, y_6^{(1)}, y_0^{(1)},y_2^{(1)},\\
	\quad  \frac{c}{\bar{c}_{3_1} \bar{c}_{3_2}}y_3^{(1)},y_4^{(1)},y_5^{(1)}), \ k=3,\\
V_2(y_5^{(3)},\bar{c}_{4_1} y_4^{(4)},  y_3^{(3)}, y_2^{(2)}, y_6^{(2)}, \bar{c}_{4_2} y_4^{(3)}, y_3^{(2)}, y_5^{(2)}, \bar{c}_{4_3} y_4^{(2)}, y_6^{(1)}, y_0^{(1)},\\
	\quad y_2^{(1)},y_3^{(1)}, \frac{c}{\bar{c}_{4_1} \bar{c}_{4_2} \bar{c}_{4_3}} y_4^{(1)},y_5^{(1)}), \ k=4,\\
V_2(\bar{c}_{5_1} y_5^{(3)},y_4^{(4)}, y_3^{(3)}, y_2^{(2)}, y_6^{(2)}, y_4^{(3)}, y_3^{(2)}, \bar{c}_{5_2} y_5^{(2)}, y_4^{(2)}, y_6^{(1)}, y_0^{(1)},y_2^{(1)},\\
	\quad y_3^{(1)},y_4^{(1)}, \frac{c}{\bar{c}_{5_1} \bar{c}_{5_2}}y_5^{(1)}), \ k=5,\\
V_2(y_5^{(3)},y_4^{(4)}, y_3^{(3)}, y_2^{(2)}, \bar{c}_6 y_6^{(2)}, y_4^{(3)}, y_3^{(2)}, y_5^{(2)}, y_4^{(2)},  \frac{c}{\bar{c}_6} y_6^{(1)}, y_0^{(1)},y_2^{(1)},\\
	\quad y_3^{(1)},y_4^{(1)},y_5^{(1)}), \ k=6,\\
\end{cases} \\
\text{where} & \\
&\bar{c}_2 =\frac{cy_2^{(2)}y_2^{(1)}+y_3^{(2)}y_0^{(1)}}{y_2^{(2)}y_2^{(1)}+y_3^{(2)}y_0^{(1)}}, \\
&\bar{c}_{3_1} = \frac{cy_3^{(3)}(y_3^{(2)})^2 y_3^{(1)} + y_4^{(3)}y_3^{(2)}y_3^{(1)}y_2^{(2)} + y_4^{(3)}y_4^{(2)}y_2^{(2)}y_2^{(1)}}{y_3^{(3)}(y_3^{(2)})^2 y_3^{(1)} + y_4^{(3)}y_3^{(2)}y_3^{(1)}y_2^{(2)} + y_4^{(3)}y_4^{(2)}y_2^{(2)}y_2^{(1)}}, \\
&\bar{c}_{3_2} = \frac{cy_3^{(3)}(y_3^{(2)})^2 y_3^{(1)} + cy_4^{(3)}y_3^{(2)}y_3^{(1)}y_2^{(2)} + y_4^{(3)}y_4^{(2)}y_2^{(2)}y_2^{(1)}}{cy_3^{(3)}(y_3^{(2)})^2 y_3^{(1)} + y_4^{(3)}y_3^{(2)}y_3^{(1)}y_2^{(2)} + y_4^{(3)}y_4^{(2)}y_2^{(2)}y_2^{(1)}}, \\
&\bar{c}_{4_1} = \frac{c\bar{c}_4^1 +  \bar{c}_4^2 +  \bar{c}_4^3+ \bar{c}_4^4}{\bar{c}_4^1 +  \bar{c}_4^2 +  \bar{c}_4^3+ \bar{c}_4^4}, \\
&\bar{c}_{4_2} = \frac{c\bar{c}_4^1 + c \bar{c}_4^2 +  \bar{c}_4^3+ \bar{c}_4^4}{c\bar{c}_4^1 +  \bar{c}_4^2 +  \bar{c}_4^3+ \bar{c}_4^4}, \\
&\bar{c}_{4_3} = \frac{c\bar{c}_4^1 + c \bar{c}_4^2 + c \bar{c}_4^3+ \bar{c}_4^4}{c\bar{c}_4^1 + c \bar{c}_4^2 +  \bar{c}_4^3+ \bar{c}_4^4}, \\
&\bar{c}_4^1 = y_4^{(4)}(y_4^{(3)})^2(y_4^{(2)})^2 y_4^{(1)},\\
&\bar{c}_4^2 = y_6^{(2)} y_4^{(3)}(y_4^{(2)})^2 y_4^{(1)}y_3^{(3)} ,\\
&\bar{c}_4^3 = y_6^{(2)}y_5^{(2)}y_4^{(2)}y_4^{(1)}y_3^{(3)}y_3^{(2)}, \\
&\bar{c}_4^4 = y_6^{(2)}y_6^{(1)}y_5^{(2)}y_3^{(3)}y_3^{(2)}y_3^{(1)},\\
&\bar{c}_{5_1} = \frac{c y_5^{(3)}(y_5^{(2)})^2 y_5^{(1)} + y_5^{(2)}y_5^{(1)}y_4^{(4)}y_4^{(3)} + y_4^{(4)}y_4^{(3)}y_4^{(2)}y_4^{(1)}}{ y_5^{(3)}(y_5^{(2)})^2 y_5^{(1)} + y_5^{(2)}y_5^{(1)}y_4^{(4)}y_4^{(3)} + y_4^{(4)}y_4^{(3)}y_4^{(2)}y_4^{(1)}}, \\
&\bar{c}_{5_2} = \frac{c y_5^{(3)}(y_5^{(2)})^2 y_5^{(1)} + c y_5^{(2)}y_5^{(1)}y_4^{(4)}y_4^{(3)} + y_4^{(4)}y_4^{(3)}y_4^{(2)}y_4^{(1)}}{ c y_5^{(3)}(y_5^{(2)})^2 y_5^{(1)} + y_5^{(2)}y_5^{(1)}y_4^{(4)}y_4^{(3)} + y_4^{(4)}y_4^{(3)}y_4^{(2)}y_4^{(1)}}, \\
&\bar{c}_6 = \frac{cy_6^{(2)}y_6^{(1)}+y_4^{(3)}y_4^{(2)}}{y_6^{(2)}y_6^{(1)}+y_4^{(3)}y_4^{(2)}}. 
\end{align*}

\begin{align*}
\bar{\gamma}_k(V_2(y)) &= 
\begin{cases}
\frac{(y_0^{(1)})^2}{y_2^{(2)}y_2^{(1)}}, &k=0,\\
\frac{(y_2^{(2)})^2 (y_2^{(1)})^2}{y_3^{(3)}y_3^{(2)}y_3^{(1)}y_0^{(1)}}, &k=2, \\ 
\frac{(y_3^{(3)})^2 (y_3^{(2)})^2 (y_3^{(1)})^2}{y_4^{(4)}y_4^{(3)}y_4^{(2)}y_4^{(1)}y_2^{(2)}y_2^{(1)}}, &k=3, \\ 
\frac{(y_4^{(4)})^2 (y_4^{(3)})^2 (y_4^{(2)})^2 (y_4^{(1)})^2}{y_6^{(2)}y_6^{(1)}y_5^{(3)}y_5^{(2)}y_5^{(1)}y_3^{(3)}y_3^{(2)}y_3^{(1)}}, &k=4, \\ 
\frac{(y_5^{(3)})^2 (y_5^{(2)})^2 (y_5^{(1)})^2}{y_4^{(4)} y_4^{(3)} y_4^{(2)}y_4^{(1)}}, &k=5, \\ 
\frac{(y_6^{(2)})^2 (y_6^{(1)})^2}{y_4^{(4)} y_4^{(3)}y_4^{(2)}y_4^{(1)}}, &k=6. \\ 
\end{cases}\\
\bar{\veps}_k(V_2(y)) &= 
\begin{cases}
\frac{y_2^{(2)}}{y_0^{(1)}}, &k=0, \\ 
\frac{y_3^{(3)}}{y_2^{(2)}} \big( 1 + \frac{y_3^{(2)}y_0^{(1)}}{y_2^{(2)}y_2^{(1)}} \big), &k=2, \\ 
\frac{y_4^{(4)}}{y_3^{(3)}} \big( 1 + \frac{y_4^{(3)}y_2^{(2)}}{y_3^{(3)}y_3^{(2)}}+ \frac{y_4^{(3)}y_4^{(2)}y_2^{(2)}y_2^{(1)}}{y_3^{(3)}(y_3^{(2)})^2 y_3^{(1)}} \big), &k=3, \\ 
\frac{y_5^{(3)}}{y_4^{(4)}} \big( 1 + \frac{y_6^{(2)}y_3^{(3)}}{y_4^{(4)}y_4^{(3)}}+ \frac{y_6^{(2)}y_5^{(2)}y_3^{(3)}y_3^{(2)}}{y_4^{(4)}(y_4^{(3)})^2 y_4^{(2)}} + \frac{y_6^{(2)}y_6^{(1)}y_5^{(2)}y_3^{(3)}y_3^{(2)}y_3^{(1)}}{y_4^{(4)}(y_4^{(3)})^2 (y_4^{(2)})^2 y_4^{(1)}} \big), &k=4, \\ 
\frac{1}{y_5^{(3)}} \big( 1 + \frac{y_4^{(4)}y_4^{(3)}}{y_5^{(3)}y_5^{(2)}}+ \frac{y_4^{(4)}y_4^{(3)}y_4^{(2)}y_4^{(1)}}{y_5^{(3)}(y_5^{(2)})^2 y_5^{(1)}} \big), &k=5, \\ 
\frac{y_4^{(4)}}{y_6^{(2)}} \big( 1 + \frac{y_4^{(3)}y_4^{(2)}}{y_6^{(2)}y_6^{(1)}} \big), &k=6. \\ 
\end{cases} 
\end{align*}

\begin{prop}\label{relation2} The relation $\bar{\sigma} e_2^c = \bar{e_4}^c \bar{\sigma}$ holds.
\begin{proof} Set $e_2^c (V_1(x)) = V_1(z)$, $\bar{\sigma}(V_1(z)) = V_2(y')$,  $\bar{\sigma} (V_1(x)) = V_2(y)$ and $\bar{e_4}^c (V_2(y)) = V_2(w)$. We need to show that $y_m^{(l)'} = w_m^{(l)}$ for $(l,m) \in \{(1,0), (1,2), (1,3), (1,4), (1,5), \\ (1,6), (2,2), (2,3), (2,4), (2,5), (2,6), (3,3), (3,4), (3,5), (4,4)\}$. Let us check this equality for $l=1$ and the rest can be verified similarly. 
\begin{itemize}
\item $y_0^{(1)'} = \frac{z_6^{(3)}z_6^{(2)}z_6^{(1)}}{z_5^{(2)}z_5^{(1)}} =\frac{x_6^{(3)}x_6^{(2)}x_6^{(1)}}{x_5^{(2)}x_5^{(1)}} = y_0^{(1)} = w_0^{(1)}.$
\item $y_2^{(1)'} = \Big( \frac{z_5^{(2)}z_5^{(1)}}{z_6^{(3)}z_6^{(2)}}+\frac{z_6^{(1)}z_5^{(2)}z_5^{(1)}}{z_6^{(3)}z_4^{(2)}z_4^{(1)}}+\frac{z_6^{(2)}					z_6^{(1)}  z_5^{(2)}z_5^{(1)}}{z_4^{(4)}z_4^{(3)}z_4^{(2)}z_4^{(1)}}\Big)^{-1} \\
			= \Big( \frac{x_5^{(2)}x_5^{(1)}}{x_6^{(3)}x_6^{(2)}}+\frac{x_6^{(1)}x_5^{(2)}x_5^{(1)}}{x_6^{(3)}x_4^{(2)}x_4^{(1)}}+\frac{x_6^{(2)}					x_6^{(1)}  x_5^{(2)}x_5^{(1)}}{x_4^{(4)}x_4^{(3)}x_4^{(2)}x_4^{(1)}}\Big)^{-1} = y_2^{(1)} = w_2^{(1)}$.
\item $y_3^{(1)'} = \Big( \frac{z_6^{(2)}z_6^{(1)}}{z_5^{(3)}z_4^{(2)}}+\frac{z_6^{(2)}z_6^{(1)}z_5^{(2)}}{z_4^{(4)}z_4^{(3)}z_4^{(2)}}+\frac{z_6^{(2)}					z_4^{(1)}}{z_5^{(3)}z_3^{(1)}}+\frac{z_6^{(2)}z_5^{(2)}z_4^{(1)}}{z_4^{(4)}z_4^{(3)}z_3^{(1)}}+\frac{z_6^{(2)}z_4^{(2)}z_4^{(1)}}					{z_4^{(4)}z_3^{(2)}z_3^{(1)}}+\frac{z_4^{(3)}z_4^{(2)}z_4^{(1)}}{z_3^{(3)}z_3^{(2)}z_3^{(1)}}\Big)^{-1} \\
			=\Big( \frac{y_6^{(2)}y_6^{(1)}}{y_5^{(3)}y_4^{(2)}}+\frac{y_6^{(2)}y_6^{(1)}y_5^{(2)}}{y_4^{(4)}y_4^{(3)}y_4^{(2)}}+\frac{y_6^{(2)}					y_4^{(1)}}{y_5^{(3)}y_3^{(1)}}+\frac{y_6^{(2)}y_5^{(2)}y_4^{(1)}}{y_4^{(4)}y_4^{(3)}y_3^{(1)}}+\frac{y_6^{(2)}y_4^{(2)}y_4^{(1)}}					{y_4^{(4)}y_3^{(2)}y_3^{(1)}}+\frac{y_4^{(3)}y_4^{(2)}y_4^{(1)}}{y_3^{(3)}y_3^{(2)}y_3^{(1)}}\Big)^{-1} \\
			= y_3^{(1)} = w_3^{(1)}$.
\item $y_4^{(1)'} = \Big( \frac{z_5^{(2)}}{z_6^{(3)}}+\frac{z_6^{(2)}z_5^{(2)}}{z_4^{(4)}z_4^{(3)}}+\frac{z_5^{(2)}z_4^{(2)}}{z_4^{(4)}z_3^{(2)}}+						\frac{z_4^{(3)}z_4^{(2)}}{z_3^{(3)}z_3^{(2)}}+\frac{z_5^{(2)}z_3^{(1)}}{z_4^{(4)}z_2^{(1)}}+\frac{z_4^{(3)}z_3^{(1)}}{z_3^{(3)}					z_2^{(1)}}+\frac{z_3^{(2)}z_3^{(1)}}{z_2^{(2)}z_2^{(1)}} \Big)^{-1}\\
			= \Big( \frac{x_5^{(2)}}{x_6^{(3)}}+\frac{x_6^{(2)}x_5^{(2)}}{x_4^{(4)}x_4^{(3)}}+\frac{x_5^{(2)}x_4^{(2)}}{x_4^{(4)}x_3^{(2)}}+						\frac{x_4^{(3)}x_4^{(2)}}{x_3^{(3)}x_3^{(2)}}+\frac{x_5^{(2)}x_3^{(1)}(cx_2^{(2)}x_2^{(1)}+x_3^{(2)}x_1^{(1)})}{x_4^{(4)}x_2^{(1)}				(cx_2^{(2)}x_2^{(1)}+cx_3^{(2)}x_1^{(1)})}\\
				+\frac{x_4^{(3)}x_3^{(1)}(cx_2^{(2)}x_2^{(1)}+x_3^{(2)}x_1^{(1)})}{x_3^{(3)}x_2^{(1)}(cx_2^{(2)}x_2^{(1)}+cx_3^{(2)}x_1^{(1)})}+				\frac{x_3^{(2)}x_3^{(1)}}{cx_2^{(2)}x_2^{(1)}} \Big)^{-1} \\
			= (cx_2^{(2)}x_2^{(1)}+cx_3^{(2)}x_1^{(1)}) \Big( \big(\frac{x_5^{(2)}}{x_6^{(3)}}+\frac{x_6^{(2)}x_5^{(2)}}{x_4^{(4)}x_4^{(3)}}+						\frac{x_5^{(2)}x_4^{(2)}}{x_4^{(4)}x_3^{(2)}}+\frac{x_4^{(3)}x_4^{(2)}}{x_3^{(3)}x_3^{(2)}}+\frac{x_3^{(2)}x_3^{(1)}}{cx_2^{(2)}					x_2^{(1)}} \big) (cx_2^{(2)}x_2^{(1)}\\
				+cx_3^{(2)}x_1^{(1)}) +{\big( \frac{x_5^{(2)}x_3^{(1)}}{x_4^{(4)}x_2^{(1)}} + \frac{x_4^{(3)}x_3^{(1)}}{x_3^{(3)}x_2^{(1)}}   \big) }				(cx_2^{(2)}x_2^{(1)}+x_3^{(2)}x_1^{(1)})      \Big)^{-1} \\
			= \Big( \frac{x_5^{(2)}}{x_6^{(3)}}+\frac{x_6^{(2)}x_5^{(2)}}{x_4^{(4)}x_4^{(3)}}+\frac{x_5^{(2)}x_4^{(2)}}{x_4^{(4)}x_3^{(2)}}+						\frac{x_4^{(3)}x_4^{(2)}}{x_3^{(3)}x_3^{(2)}}+\frac{x_5^{(2)}x_3^{(1)}}{x_4^{(4)}x_2^{(1)}}+\frac{x_4^{(3)}x_3^{(1)}}{x_3^{(3)}					x_2^{(1)}}+\frac{x_3^{(2)}x_3^{(1)}}{x_2^{(2)}x_2^{(1)}} \Big)^{-1} \\
				\times \Big( \big(  \frac{x_5^{(2)}}{x_6^{(3)}}+\frac{x_6^{(2)}x_5^{(2)}}{x_4^{(4)}x_4^{(3)}}+\frac{x_5^{(2)}x_4^{(2)}}{x_4^{(4)}x_3^{(2)}}+				\frac{x_4^{(3)}x_4^{(2)}}{x_3^{(3)}x_3^{(2)}}+\frac{x_5^{(2)}x_3^{(1)}}{x_4^{(4)}x_2^{(1)}}+\frac{x_4^{(3)}x_3^{(1)}}{x_3^{(3)}					x_2^{(1)}}+\frac{x_3^{(2)}x_3^{(1)}}{x_2^{(2)}x_2^{(1)}} \big) (cx_2^{(2)}x_2^{(1)}\\
				+cx_3^{(2)}x_1^{(1)})\Big) \Big( \big(\frac{x_5^{(2)}}{x_6^{(3)}}+\frac{x_6^{(2)}x_5^{(2)}}{x_4^{(4)}x_4^{(3)}}+\frac{x_5^{(2)}						x_4^{(2)}}{x_4^{(4)}x_3^{(2)}}+\frac{x_4^{(3)}x_4^{(2)}}{x_3^{(3)}x_3^{(2)}}+\frac{x_3^{(2)}x_3^{(1)}}{cx_2^{(2)}x_2^{(1)}} \big) 					(cx_2^{(2)}x_2^{(1)} \\
				+cx_3^{(2)}x_1^{(1)}) 
				+{\big( \frac{x_5^{(2)}x_3^{(1)}}{x_4^{(4)}x_2^{(1)}} + \frac{x_4^{(3)}x_3^{(1)}}{x_3^{(3)}x_2^{(1)}}   \big) }(cx_2^{(2)}x_2^{(1)}					+x_3^{(2)}x_1^{(1)})      \Big)^{-1}\\
	= y_4^{(1)} \Big( cy_4^{(4)}(y_4^{(3)})^2(y_4^{(2)})^2 y_4^{(1)} + cy_6^{(2)} y_4^{(3)}(y_4^{(2)})^2 y_4^{(1)}y_3^{(3)} + cy_6^{(2)}y_5^{(2)}y_4^{(2)}		y_4^{(1)}y_3^{(3)}y_3^{(2)}  \\
		+ cy_6^{(2)}y_6^{(1)}y_5^{(2)}y_3^{(3)}y_3^{(2)}y_3^{(1)}\Big) \Big(cy_4^{(4)}(y_4^{(3)})^2(y_4^{(2)})^2 y_4^{(1)} +c y_6^{(2)} y_4^{(3)}				(y_4^{(2)})^2 y_4^{(1)}y_3^{(3)} \\
		+ cy_6^{(2)}y_5^{(2)}y_4^{(2)}y_4^{(1)}y_3^{(3)}y_3^{(2)} + y_6^{(2)}y_6^{(1)}y_5^{(2)}y_3^{(3)}y_3^{(2)}y_3^{(1)} \Big)^{-1}  \ \text{by Maple}\\
	= w_4^{(1)}$.
\item $y_5^{(1)'} = \Big( \frac{z_5^{(2)}}{z_4^{(4)}}+\frac{z_4^{(3)}}{z_3^{(3)}}+\frac{z_3^{(2)}}{z_2^{(2)}}+\frac{z_2^{(1)}}{z_1^{(1)}} \Big)^{-1} \\
	= \Big( \frac{x_5^{(2)}}{x_4^{(4)}}+\frac{x_4^{(3)}}{x_3^{(3)}}+\frac{x_3^{(2)}(x_2^{(2)}x_2^{(1)}+x_3^{(2)}x_1^{(1)})}{x_2^{(2)}(cx_2^{(2)}x_2^{(1)}+x_3^{(2)}x_1^{(1)})}+\frac{x_2^{(1)}(cx_2^{(2)}x_2^{(1)}+cx_3^{(2)}x_1^{(1)})}{x_1^{(1)}(cx_2^{(2)}x_2^{(1)}+x_3^{(2)}x_1^{(1)})} \Big)^{-1}\\
	= \Big( \frac{x_5^{(2)}}{x_4^{(4)}}+\frac{x_4^{(3)}}{x_3^{(3)}}+\frac{x_3^{(2)}}{x_2^{(2)}}+\frac{x_2^{(1)}}{x_1^{(1)}} \Big)^{-1} = y_5^{(1)} = w_5^{(1)}$.
\item $y_6^{(1)'} = \frac{1}{z_5^{(2)}} = \frac{1}{x_5^{(2)}} = y_6^{(1)} = w_6^{(1)}$.
\end{itemize}
\end{proof}
\end{prop}
\section {$D_6^{(1)}$- Geometric Crystal $\mathcal{V}$}
In order to give $\mathcal{V}_1$ a $\ge=D_6^{(1)}$-geometric crystal structure, we need to define the actions of $e_0^c, \gamma_0$, and $\veps_0$ on $V_1(x)$. We use the $\ge_1$-geometric crystal structure on $\cV_2$ to define the action of $e_0^c$, $\gamma_0$, and $\veps_0$ on $V_1(x)$ as follows. 
\begin{align}
e_0^c (V_1(x)) &:= \bar{\sigma}^{-1} \circ \overline{e_{\sigma{(0)}}}^c \circ \bar{\sigma} (V_1(x))= \bar{\sigma}^{-1} \circ \bar{e}_6^c(V_2(y)), \\
\gamma_0 (V_1(x)) &:= \overline{ \gamma_{\sigma{(0)}}} \circ \bar{\sigma} (V_1(x))= \overline{ \gamma_6} (V_2(y)). \label{gamma_0}\\
\veps_0 (V_1(x)) &:= \overline{ \veps_{\sigma{(0)}}} \circ \bar{\sigma} (V_1(x))= \overline{ \veps_6}(V_2(y)), \label{veps_0}
\end{align}

Set \begin{align*} 
K= K_x &= x_6^{(1)} + \frac{x_5^{(1)} x_2^{(2)} x_2^{(1)}}{x_3^{(3)} x_3^{(2)}}+ \frac{x_5^{(1)} x_3^{(1)} x_2^{(2)}}{x_4^{(2)} x_3^{(3)}}+ \frac{x_4^{(1)} 
		x_2^{(2)}}{x_3^{(3)} }+ \frac{x_5^{(1)} x_3^{(2)} x_3^{(1)}}{x_4^{(3)} x_4^{(2)}}+ \frac{x_4^{(1)} x_3^{(2)}}{x_4^{(3)}}  \\
		&\quad+\frac{x_4^{(2)} x_4^{(1)}}{ x_6^{(2)}}+ \frac{x_2^{(2)} x_2^{(1)}}{x_6^{(3)} } + \frac{x_6^{(2)} x_2^{(2)} x_2^{(1)}}{x_4^{(4)} x_4^{(3)}}+ 
		\frac{x_4^{(2)} x_2^{(2)} x_2^{(1)}}{x_4^{(4)} x_3^{(2)}} + \frac{x_4^{(3)} x_4^{(2)} x_2^{(2)} x_2^{(1)}}{x_5^{(2)} x_3^{(3)} x_3^{(2)}} \\
		&\quad+ \frac{x_3^{(1)} x_2^{(2)} }{x_4^{(4)}} + \frac{x_4^{(3)} x_3^{(1)} x_2^{(2)}}{x_5^{(2)} x_3^{(3)}} + \frac{x_3^{(2)} x_3^{(1)}}{ x_5^{(2)}},  \\
K_{3_1} = K_{x 3_1}&= c x_6^{(1)} +  \frac{x_5^{(1)} x_2^{(2)} x_2^{(1)}}{x_3^{(3)} x_3^{(2)}}+ c \frac{x_5^{(1)} x_3^{(1)} x_2^{(2)}}{x_4^{(2)} x_3^{(3)}}+ c 
		\frac{x_4^{(1)} x_2^{(2)}}{x_3^{(3)} }+ c  \frac{x_5^{(1)} x_3^{(2)} x_3^{(1)}}{x_4^{(3)} x_4^{(2)}}+ c \frac{x_4^{(1)} x_3^{(2)}}{x_4^{(3)}}  \\
		&\quad +c \frac{x_4^{(2)} x_4^{(1)}}{ x_6^{(2)}}+ \frac{x_2^{(2)} x_2^{(1)}}{x_6^{(3)} } + \frac{x_6^{(2)} x_2^{(2)} x_2^{(1)}}{x_4^{(4)} x_4^{(3)}}+ 
		\frac{x_4^{(2)} x_2^{(2)} x_2^{(1)}}{x_4^{(4)} x_3^{(2)}} + \frac{x_4^{(3)} x_4^{(2)} x_2^{(2)} x_2^{(1)}}{x_5^{(2)} x_3^{(3)} x_3^{(2)}} \\
		&\quad +c \frac{x_3^{(1)} x_2^{(2)} }{x_4^{(4)}} + c \frac{x_4^{(3)} x_3^{(1)} x_2^{(2)}}{x_5^{(2)} x_3^{(3)}} + c \frac{x_3^{(2)} x_3^{(1)}}	
		{ x_5^{(2)}},  \\
K_{3_2} =K_{x 3_2}&= c x_6^{(1)} + \frac{x_5^{(1)} x_2^{(2)} x_2^{(1)}}{x_3^{(3)} x_3^{(2)}}+ \frac{x_5^{(1)} x_3^{(1)} x_2^{(2)}}{x_4^{(2)} x_3^{(3)}}+ 		
		\frac{x_4^{(1)} x_2^{(2)}}{x_3^{(3)} }+ c \frac{x_5^{(1)} x_3^{(2)} x_3^{(1)}}{x_4^{(3)} x_4^{(2)}}+ c \frac{x_4^{(1)} x_3^{(2)}}{x_4^{(3)}}  \\
		&\quad+c \frac{x_4^{(2)} x_4^{(1)}}{ x_6^{(2)}}+ \frac{x_2^{(2)} x_2^{(1)}}{x_6^{(3)} } + \frac{x_6^{(2)} x_2^{(2)} x_2^{(1)}}{x_4^{(4)} x_4^{(3)}}+ 
		\frac{x_4^{(2)} x_2^{(2)} x_2^{(1)}}{x_4^{(4)} x_3^{(2)}} + \frac{x_4^{(3)} x_4^{(2)} x_2^{(2)} x_2^{(1)}}{x_5^{(2)} x_3^{(3)} x_3^{(2)}} \\
		&\quad+ \frac{x_3^{(1)} x_2^{(2)} }{x_4^{(4)}} + \frac{x_4^{(3)} x_3^{(1)} x_2^{(2)}}{x_5^{(2)} x_3^{(3)}} + c \frac{x_3^{(2)} x_3^{(1)}}{ x_5^{(2)}},  \\
K_{4_1} =K_{x 4_1}&= c x_6^{(1)} + \frac{x_5^{(1)} x_2^{(2)} x_2^{(1)}}{x_3^{(3)} x_3^{(2)}}+ \frac{x_5^{(1)} x_3^{(1)} x_2^{(2)}}{x_4^{(2)} x_3^{(3)}}+ 
		c \frac{x_4^{(1)}x_2^{(2)}}{x_3^{(3)} }+ \frac{x_5^{(1)} x_3^{(2)} x_3^{(1)}}{x_4^{(3)} x_4^{(2)}}+ c \frac{x_4^{(1)} x_3^{(2)}}{x_4^{(3)}}  \\
		&\quad+c \frac{x_4^{(2)} x_4^{(1)}}{ x_6^{(2)}}+ \frac{x_2^{(2)} x_2^{(1)}}{x_6^{(3)} } + \frac{x_6^{(2)} x_2^{(2)} x_2^{(1)}}{x_4^{(4)} x_4^{(3)}}+ 
		\frac{x_4^{(2)} x_2^{(2)} x_2^{(1)}}{x_4^{(4)} x_3^{(2)}} + \frac{x_4^{(3)} x_4^{(2)} x_2^{(2)} x_2^{(1)}}{x_5^{(2)} x_3^{(3)} x_3^{(2)}} \\
		&\quad +\frac{x_3^{(1)} x_2^{(2)} }{x_4^{(4)}} + \frac{x_4^{(3)} x_3^{(1)} x_2^{(2)}}{x_5^{(2)} x_3^{(3)}} + \frac{x_3^{(2)} x_3^{(1)}}{ x_5^{(2)}},  \\
K_{4_2} =K_{x 4_2}&= cx_6^{(1)} \big(cx_6^{(1)} + c\frac{x_5^{(1)} x_2^{(2)} x_2^{(1)}}{x_3^{(3)} x_3^{(2)}}+ c\frac{x_5^{(1)} x_3^{(1)} x_2^{(2)}}{x_4^{(2)} x_3^{(3)}}+ c\frac{x_4^{(1)} 
		x_2^{(2)}}{x_3^{(3)} }+ c\frac{x_5^{(1)} x_3^{(2)} x_3^{(1)}}{x_4^{(3)} x_4^{(2)}} \\
		&\quad+ c\frac{x_4^{(1)} x_3^{(2)}}{x_4^{(3)}} +c\frac{x_4^{(2)} x_4^{(1)}}{ x_6^{(2)}}+ \frac{x_2^{(2)} x_2^{(1)}}{x_6^{(3)} } + \frac{x_6^{(2)} x_2^{(2)} x_2^{(1)}}{x_4^{(4)} x_4^{(3)}}+ 
		c\frac{x_4^{(2)} x_2^{(2)} x_2^{(1)}}{x_4^{(4)} x_3^{(2)}}\\
		&\quad + c\frac{x_4^{(3)} x_4^{(2)} x_2^{(2)} x_2^{(1)}}{x_5^{(2)} x_3^{(3)} x_3^{(2)}} + c\frac{x_3^{(1)} x_2^{(2)} }{x_4^{(4)}} + c\frac{x_4^{(3)} x_3^{(1)} x_2^{(2)}}{x_5^{(2)} x_3^{(3)}} + c\frac{x_3^{(2)} x_3^{(1)}}{ x_5^{(2)}} \big)+ c\frac{x_4^{(2)} x_4^{(1)}}{ x_6^{(2)}} \big(cx_6^{(1)}  \\
		& \quad + c\frac{x_5^{(1)} x_2^{(2)} x_2^{(1)}}{x_3^{(3)} x_3^{(2)}}+ c\frac{x_5^{(1)} x_3^{(1)} x_2^{(2)}}{x_4^{(2)} x_3^{(3)}}+ c\frac{x_4^{(1)} x_2^{(2)}}{x_3^{(3)} }+ c\frac{x_5^{(1)} x_3^{(2)} x_3^{(1)}}{x_4^{(3)} x_4^{(2)}}+ c\frac{x_4^{(1)} x_3^{(2)}}{x_4^{(3)}} \\
		&\quad+c\frac{x_4^{(2)} x_4^{(1)}}{ x_6^{(2)}}+ \frac{x_2^{(2)} x_2^{(1)}}{x_6^{(3)} } +c \frac{x_6^{(2)} x_2^{(2)} x_2^{(1)}}{x_4^{(4)} x_4^{(3)}}+ 
		c\frac{x_4^{(2)} x_2^{(2)} x_2^{(1)}}{x_4^{(4)} x_3^{(2)}}+c \frac{x_4^{(3)} x_4^{(2)} x_2^{(2)} x_2^{(1)}}{x_5^{(2)} x_3^{(3)} x_3^{(2)}}   \\
		&\quad+ c\frac{x_3^{(1)} x_2^{(2)} }{x_4^{(4)}} +c \frac{x_4^{(3)} x_3^{(1)} x_2^{(2)}}{x_5^{(2)} x_3^{(3)}} + c\frac{x_3^{(2)} x_3^{(1)}}{ x_5^{(2)}}\big) + \frac{x_2^{(2)} x_2^{(1)}}{x_6^{(3)} }  \big(cx_6^{(1)} + \frac{x_5^{(1)} x_2^{(2)} x_2^{(1)}}{x_3^{(3)} x_3^{(2)}}\\
		& \quad + \frac{x_5^{(1)} x_3^{(1)} x_2^{(2)}}{x_4^{(2)} x_3^{(3)}}+ \frac{x_4^{(1)} 
		x_2^{(2)}}{x_3^{(3)} }+ \frac{x_5^{(1)} x_3^{(2)} x_3^{(1)}}{x_4^{(3)} x_4^{(2)}}+ \frac{x_4^{(1)} x_3^{(2)}}{x_4^{(3)}} +c\frac{x_4^{(2)} x_4^{(1)}}{ x_6^{(2)}}+ \frac{x_2^{(2)} x_2^{(1)}}{x_6^{(3)} } \\
		&\quad+ \frac{x_6^{(2)} x_2^{(2)} x_2^{(1)}}{x_4^{(4)} x_4^{(3)}} + \frac{x_4^{(2)} x_2^{(2)} x_2^{(1)}}{x_4^{(4)} x_3^{(2)}} + \frac{x_4^{(3)} x_4^{(2)} x_2^{(2)} x_2^{(1)}}{x_5^{(2)} x_3^{(3)} x_3^{(2)}} + \frac{x_3^{(1)} x_2^{(2)} }{x_4^{(4)}} + \frac{x_4^{(3)} x_3^{(1)} x_2^{(2)}}{x_5^{(2)} x_3^{(3)}} \\
		&\quad+ \frac{x_3^{(2)} x_3^{(1)}}{ x_5^{(2)}} \big)+ \frac{x_6^{(2)} x_2^{(2)} x_2^{(1)}}{x_4^{(4)} x_4^{(3)}} \big(cx_6^{(1)} + \frac{x_5^{(1)} x_2^{(2)} x_2^{(1)}}{x_3^{(3)} x_3^{(2)}}+ \frac{x_5^{(1)} x_3^{(1)} x_2^{(2)}}{x_4^{(2)} x_3^{(3)}}+ \frac{x_4^{(1)} 
		x_2^{(2)}}{x_3^{(3)} }\\
		&\quad+ \frac{x_5^{(1)} x_3^{(2)} x_3^{(1)}}{x_4^{(3)} x_4^{(2)}}+ \frac{x_4^{(1)} x_3^{(2)}}{x_4^{(3)}} +\frac{x_4^{(2)} x_4^{(1)}}{ x_6^{(2)}}+ \frac{x_2^{(2)} x_2^{(1)}}{x_6^{(3)} }  + \frac{x_6^{(2)} x_2^{(2)} x_2^{(1)}}{x_4^{(4)} x_4^{(3)}}+ \frac{x_4^{(2)} x_2^{(2)} x_2^{(1)}}{x_4^{(4)} x_3^{(2)}} \\
		&\quad+ \frac{x_4^{(3)} x_4^{(2)} x_2^{(2)} x_2^{(1)}}{x_5^{(2)} x_3^{(3)} x_3^{(2)}} + \frac{x_3^{(1)} x_2^{(2)} }{x_4^{(4)}} + \frac{x_4^{(3)} x_3^{(1)} x_2^{(2)}}{x_5^{(2)} x_3^{(3)}} + \frac{x_3^{(2)} x_3^{(1)}}{ x_5^{(2)}} \big) + \big(c\frac{x_5^{(1)} x_2^{(2)} x_2^{(1)}}{x_3^{(3)} x_3^{(2)}}\\
		& \quad + c\frac{x_5^{(1)} x_3^{(1)} x_2^{(2)}}{x_4^{(2)} x_3^{(3)}}+ c\frac{x_4^{(1)} 
		x_2^{(2)}}{x_3^{(3)} }+ c\frac{x_5^{(1)} x_3^{(2)} x_3^{(1)}}{x_4^{(3)} x_4^{(2)}}+ c\frac{x_4^{(1)} x_3^{(2)}}{x_4^{(3)}} + c\frac{x_4^{(2)} x_2^{(2)} x_2^{(1)}}{x_4^{(4)} x_3^{(2)}}  \\
		&\quad+ c\frac{x_4^{(3)} x_4^{(2)} x_2^{(2)} x_2^{(1)}}{x_5^{(2)} x_3^{(3)} x_3^{(2)}} + c\frac{x_3^{(1)} x_2^{(2)} }{x_4^{(4)}} + c\frac{x_4^{(3)} x_3^{(1)} x_2^{(2)}}{x_5^{(2)} x_3^{(3)}} + c\frac{x_3^{(2)} x_3^{(1)}}{ x_5^{(2)}} \big)K,\\
K_{4_3} =K_{x 4_3}&= c x_6^{(1)} +  c \frac{x_5^{(1)} x_2^{(2)} x_2^{(1)}}{x_3^{(3)} x_3^{(2)}}+ c \frac{x_5^{(1)} x_3^{(1)} x_2^{(2)}}{x_4^{(2)} x_3^{(3)}}+ c 
		\frac{x_4^{(1)} x_2^{(2)}}{x_3^{(3)} }+ c  \frac{x_5^{(1)} x_3^{(2)} x_3^{(1)}}{x_4^{(3)} x_4^{(2)}}+ c \frac{x_4^{(1)} x_3^{(2)}}{x_4^{(3)}}  \\
		&\quad+ c \frac{x_4^{(2)} x_4^{(1)}}{ x_6^{(2)}}+ \frac{x_2^{(2)} x_2^{(1)}}{x_6^{(3)} } +  \frac{x_6^{(2)} x_2^{(2)} x_2^{(1)}}{x_4^{(4)} x_4^{(3)}}+ 
		 \frac{x_4^{(2)} x_2^{(2)} x_2^{(1)}}{x_4^{(4)} x_3^{(2)}} + c \frac{x_4^{(3)} x_4^{(2)} x_2^{(2)} x_2^{(1)}}{x_5^{(2)} x_3^{(3)} x_3^{(2)}} \\
		&\quad + \frac{x_3^{(1)} x_2^{(2)} }{x_4^{(4)}} + c \frac{x_4^{(3)} x_3^{(1)} x_2^{(2)}}{x_5^{(2)} x_3^{(3)}} + c \frac{x_3^{(2)} x_3^{(1)}}	
		{ x_5^{(2)}},\\
K_{5_1} =K_{x 5_1}&=c  x_6^{(1)} + c \frac{x_5^{(1)} x_2^{(2)} x_2^{(1)}}{x_3^{(3)} x_3^{(2)}}+ c \frac{x_5^{(1)} x_3^{(1)} x_2^{(2)}}{x_4^{(2)} x_3^{(3)}}+
		c \frac{x_4^{(1)}x_2^{(2)}}{x_3^{(3)} }+ c \frac{x_5^{(1)} x_3^{(2)} x_3^{(1)}}{x_4^{(3)} x_4^{(2)}}+ c \frac{x_4^{(1)} x_3^{(2)}}{x_4^{(3)}}  \\
		&\quad+c \frac{x_4^{(2)} x_4^{(1)}}{ x_6^{(2)}}+ \frac{x_2^{(2)} x_2^{(1)}}{x_6^{(3)} } + \frac{x_6^{(2)} x_2^{(2)} x_2^{(1)}}{x_4^{(4)} x_4^{(3)}}+ 
		\frac{x_4^{(2)} x_2^{(2)} x_2^{(1)}}{x_4^{(4)} x_3^{(2)}} + \frac{x_4^{(3)} x_4^{(2)} x_2^{(2)} x_2^{(1)}}{x_5^{(2)} x_3^{(3)} x_3^{(2)}} \\
		&\quad +\frac{x_3^{(1)} x_2^{(2)} }{x_4^{(4)}} + \frac{x_4^{(3)} x_3^{(1)} x_2^{(2)}}{x_5^{(2)} x_3^{(3)}} + \frac{x_3^{(2)} x_3^{(1)}}{ x_5^{(2)}},  \\
K_{6_1} =K_{x 6_1}&= c x_6^{(1)} + \frac{x_5^{(1)} x_2^{(2)} x_2^{(1)}}{x_3^{(3)} x_3^{(2)}}+ \frac{x_5^{(1)} x_3^{(1)} x_2^{(2)}}{x_4^{(2)} x_3^{(3)}}+ 	
		\frac{x_4^{(1)}x_2^{(2)}}{x_3^{(3)} }+ \frac{x_5^{(1)} x_3^{(2)} x_3^{(1)}}{x_4^{(3)} x_4^{(2)}}+ \frac{x_4^{(1)} x_3^{(2)}}{x_4^{(3)}}  \\
		&\quad+\frac{x_4^{(2)} x_4^{(1)}}{ x_6^{(2)}}+ \frac{x_2^{(2)} x_2^{(1)}}{x_6^{(3)} } + \frac{x_6^{(2)} x_2^{(2)} x_2^{(1)}}{x_4^{(4)} x_4^{(3)}}+ 
		\frac{x_4^{(2)} x_2^{(2)} x_2^{(1)}}{x_4^{(4)} x_3^{(2)}} + \frac{x_4^{(3)} x_4^{(2)} x_2^{(2)} x_2^{(1)}}{x_5^{(2)} x_3^{(3)} x_3^{(2)}} \\
		&\quad +\frac{x_3^{(1)} x_2^{(2)} }{x_4^{(4)}} + \frac{x_4^{(3)} x_3^{(1)} x_2^{(2)}}{x_5^{(2)} x_3^{(3)}} + \frac{x_3^{(2)} x_3^{(1)}}{ x_5^{(2)}},  \\
K_{6_2} =K_{x 6_2}&= c x_6^{(1)} +  c \frac{x_5^{(1)} x_2^{(2)} x_2^{(1)}}{x_3^{(3)} x_3^{(2)}}+ c \frac{x_5^{(1)} x_3^{(1)} x_2^{(2)}}{x_4^{(2)} x_3^{(3)}}+ c 
		\frac{x_4^{(1)} x_2^{(2)}}{x_3^{(3)} }+ c  \frac{x_5^{(1)} x_3^{(2)} x_3^{(1)}}{x_4^{(3)} x_4^{(2)}}+ c \frac{x_4^{(1)} x_3^{(2)}}{x_4^{(3)}}  \\
		&\quad +c \frac{x_4^{(2)} x_4^{(1)}}{ x_6^{(2)}}+ \frac{x_2^{(2)} x_2^{(1)}}{x_6^{(3)} } + c \frac{x_6^{(2)} x_2^{(2)} x_2^{(1)}}{x_4^{(4)} x_4^{(3)}}+ 
		c \frac{x_4^{(2)} x_2^{(2)} x_2^{(1)}}{x_4^{(4)} x_3^{(2)}} + c \frac{x_4^{(3)} x_4^{(2)} x_2^{(2)} x_2^{(1)}}{x_5^{(2)} x_3^{(3)} x_3^{(2)}} \\
		&\quad +c \frac{x_3^{(1)} x_2^{(2)} }{x_4^{(4)}} + c \frac{x_4^{(3)} x_3^{(1)} x_2^{(2)}}{x_5^{(2)} x_3^{(3)}} + c \frac{x_3^{(2)} x_3^{(1)}}	
		{ x_5^{(2)}}.
\end{align*} The following is the main result in this paper.

\begin{theorem} The algebraic variety $\cV = \cV_1 = \{V_1(x), e_k^c, \gamma_k, \veps_k \mid k \in I\}$ is a positive geometric crystal for the affine Lie algebra $\ge = D_6^{(1)}$ with the $e_0^c$, $\gamma_0$, and $\veps_0$ actions on $V_1(x)$ given by:
\begin{align*}
\gamma_0(V_1(x)) &= \frac{1}{x_2^{(2)}x_2^{(1)}},   \\
\veps_0(V_1(x)) &=  K,  \\
e_0^{c}(V_1(x)) &= V_1(x') =  V_1(x_6^{(3)'},x_4^{(4)'}, \ldots, x_6^{(1)'}) \ \text{where}
\end{align*}
\begin{align*} 
x_1^{(1)'} &= \frac{x_1^{(1)}}{c}, \qquad &x_2^{(1)'} &= \frac{x_2^{(1)}}{c}, \qquad &x_2^{(2)'} &= \frac{x_2^{(2)}}{c},\\
x_3^{(1)'} &= x_3^{(1)} \frac{K}{K_{3_1}}, \qquad &x_3^{(2)'} &= x_3^{(2)} \frac{K_{3_1}}{c K_{3_2}}, \qquad &x_3^{(3)'} &= x_3^{(3)} \frac{K_{3_2}}{cK},\\
x_4^{(1)'} &= x_4^{(1)} \frac{K}{K_{4_1}}, \qquad &x_4^{(2)'} &= x_4^{(2)} \frac{K K_{4_1}}{K_{4_2}}, \qquad &x_4^{(3)'} &= x_4^{(3)} \frac{K_{4_2}}{c K K_{4_3} }, \\
x_4^{(4)'} &= x_4^{(4)} \frac{K_{4_3}}{cK},\qquad & x_5^{(1)'} &= x_5^{(1)} \frac{K}{K_{5_1}}, \qquad &x_5^{(2)'} &= x_5^{(2)} \frac{K_{5_1}}{c K}, \\
x_6^{(1)'} &= x_6^{(1)} \frac{K}{K_{6_1}}, \qquad &x_6^{(2)'} &= x_6^{(2)} \frac{K_{6_1}}{K_{6_2}}, \qquad &x_6^{(3)'} &= x_6^{(3)} \frac{K_{6_2}}{cK}.
\end{align*}
\end{theorem}

\begin{proof} Since $\cV= \cV_1$ is a positive geometric crystal for $\ge_0$, to show that it is a positive geometric crystal for 
$\ge = D_6^{(1)}$, it suffices to show that the following relations involving the $0$-action hold (see Definition \ref{geometric}):
\begin{enumerate}
	\item $\gamma_0 (e_k^c (V_1(x))) = c^{a_{k0}} \gamma_0 (V_1(x))$ for all $k \in I$ 
	\item $\gamma_k (e_0^c (V_1(x))) = c^{a_{0k}} \gamma_k (V_1(x))$ for all $k \in I$ 
	\item $\veps_0 (e_0^c (V_1(x))) = c^{-1} \veps_0 (V_1(x))$
	\item $e_0^{c_1} e_k^{c_2} = e_k^{c_2} e_0^{c_1}$ for all $k \in \{1,3,4,5,6\}$ 
	\item $e_0^{c_1} e_2^{c_1 c_2} e_0^{c_2}= e_2^{c_2} e_0^{c_1 c_2} e_2^{c_1}$ 
\end{enumerate}
Note that \ $x_1^{(1)'}= c^{-1}x_1^{(1)}, \ x_2^{(2)'}x_2^{(1)'}= c^{-2}x_2^{(2)}x_2^{(1)}, \  x_3^{(3)'}x_3^{(2)'}x_3^{(1)'}= c^{-2}x_3^{(3)}x_3^{(2)}x_3^{(1)}, \\  x_4^{(4)'}x_4^{(3)'}x_4^{(2)'}x_4^{(1)'}
= c^{-2}x_4^{(4)}x_4^{(3)}x_4^{(2)}x_4^{(1)}, \ x_5^{(2)'}x_5^{(1)'}= c^{-1}x_5^{(2)}x_5^{(1)}, \ x_6^{(3)'}x_6^{(2)'}x_6^{(1)'}\\= c^{-1}x_6^{(3)}x_6^{(2)}x_6^{(1)}$. Relations (1) and (2) follows easily from the defined actions. For example,
\begin{align*}
\gamma_0 (e_0^c (V_1(x))) &= \frac{1}{x_2^{(2)'}x_2^{(1)'}} =  \frac{1}{c^{-2}x_2^{(2)}x_2^{(1)}} = c^2 \gamma_0 (V_1(x)) = c^{a_{0,0}} \gamma_0 (V_1(x)),\\
\gamma_0 (e_2^c (V_1(x))) &= \frac{1}{c_2x_2^{(2)}cc_2^{-1}x_2^{(1)}} =  c^{-1} \gamma_0 (V_1(x))) = c^{a_{2,0}} \gamma_0 (V_1(x)),\\
\gamma_1 (e_0^c (V_1(x))) &= \frac{(x_1^{(1)'})^2}{x_2^{(2)'}x_2^{(1)'}} = \frac{(c^{-1})^2(x_1^{(1)})^2}{c^{-2}x_2^{(2)}x_2^{(1)}} =  c^0 \gamma_1 (V_1(x))= c^{a_{0,1}} \gamma_1 (V_1(x)).\\
\end{align*}
Now we consider relation (3). We have
\begin{align*}
	&\veps_0 (e_0^c (V_1(x)))  = x_6^{(1)'} + \frac{x_5^{(1)'} x_2^{(2)'} x_2^{(1)'}}{x_3^{(3)'} x_3^{(2)'}}+ \frac{x_5^{(1)'} x_3^{(1)'} x_2^{(2)'}}{x_4^{(2)'} x_3^{(3)'}}+ \frac{x_4^{(1)'} x_2^{(2)'}}{x_3^{(3)'} }+ \frac{x_5^{(1)'} x_3^{(2)'} x_3^{(1)'}}{x_4^{(3)'} x_4^{(2)'}}\\
	& \qquad +\frac{x_4^{(1)'} x_3^{(2)'}}{x_4^{(3)'}} + \frac{x_4^{(2)'} x_4^{(1)'}}{ x_6^{(2)'}} + \frac{x_2^{(2)'} x_2^{(1)'}}{x_6^{(3)'} } + \frac{x_6^{(2)'} x_2^{(2)'} x_2^{(1)'}}{x_4^{(4)'} x_4^{(3)'}}+ \frac{x_4^{(2)'} x_2^{(2)'} x_2^{(1)'}}{x_4^{(4)'} x_3^{(2)'}}  \\
	& \qquad +\frac{x_4^{(3)'} x_4^{(2)'} x_2^{(2)'} x_2^{(1)'}}{x_5^{(2)'} x_3^{(3)'} x_3^{(2)'}} +\frac{x_3^{(1)'} x_2^{(2)'} }{x_4^{(4)'}}+ \frac{x_4^{(3)'} x_3^{(1)'} x_2^{(2)'}}{x_5^{(2)'} x_3^{(3)'}}   + \frac{x_3^{(2)'} x_3^{(1)'}}{ x_5^{(2)'}}\\	
   	& \quad = x_6^{(1)} \cdot \frac{K}{K_{6_1}} + \frac{x_5^{(1)} x_2^{(2)} x_2^{(1)}}{x_3^{(3)} x_3^{(2)}} \cdot \frac{K^2}{K_{5_1} K_{3_1}}  
	+ \frac{x_5^{(1)} x_3^{(1)} x_2^{(2)}}{x_4^{(2)} x_3^{(3)}} \cdot  \frac{K^2 K_{4_2}}{ K_{5_1} K_{4_1} K_{3_2} K_{3_1}} \\
&\qquad + \frac{x_4^{(1)} x_2^{(2)}}{x_3^{(3)} } \cdot \frac{K^2}{ K_{4_1} K_{3_2}}+ \frac{x_5^{(1)} x_3^{(2)} x_3^{(1)}}{x_4^{(3)} x_4^{(2)}} \cdot  \frac{K^2 K_{4_3}}{ K_{5_1} K_{4_1} K_{3_2}} 
	+\frac{x_4^{(1)} x_3^{(2)}}{x_4^{(3)}}  \cdot  \frac{K^2 K_{4_3} K_{3_1}}{K_{4_2} K_{4_1} K_{3_2}}\\
&\qquad	+ \frac{x_4^{(2)} x_4^{(1)}}{ x_6^{(2)}}  \cdot  \frac{K^2 K_{6_2}}{K_{6_1} K_{4_2}}
	 + \frac{x_2^{(2)} x_2^{(1)}}{x_6^{(3)} }  \cdot \frac{K}{c K_{6_2}}  + \frac{x_6^{(2)} x_2^{(2)} x_2^{(1)}}{x_4^{(4)} x_4^{(3)}} \cdot \frac{K^2 K_{6_1}}{ K_{6_2} K_{4_2} } \\
&\qquad	 + \frac{x_4^{(2)} x_2^{(2)} x_2^{(1)}}{x_4^{(4)} x_3^{(2)}}  \cdot \frac{K^2 K_{4_1} K_{3_2}}{K_{4_3} K_{4_2} K_{3_1}}
	+\frac{x_4^{(3)} x_4^{(2)} x_2^{(2)} x_2^{(1)}}{x_5^{(2)} x_3^{(3)} x_3^{(2)}} \cdot \frac{K^2 K_{4_1}}{ K_{5_1} K_{4_3} K_{3_1}} \\
&\qquad+\frac{x_3^{(1)} x_2^{(2)} }{x_4^{(4)}} \cdot \frac{K^2}{ K_{4_3} K_{3_1}} 
	+ \frac{x_4^{(3)} x_3^{(1)} x_2^{(2)}}{x_5^{(2)} x_3^{(3)}}  \cdot  \frac{K^2 K_{4_2}}{  K_{5_1} K_{4_3} K_{3_2}  K_{3_1} }+ \frac{x_3^{(2)} x_3^{(1)}}{ x_5^{(2)}} \cdot \frac{K^2}{ K_{5_1} K_{3_2}} \\	
	& \quad = c^{-1} K \Big(\big(K_{6_2} K_{6_1} K_{5_2} K_{5_1} K_{4_3} K_{4_2} K_{4_1} K_{3_2} K_{3_1}  \big) \\
&\qquad	\times \big(K_{6_2} K_{6_1} K_{5_2} K_{5_1} K_{4_3} K_{4_2} K_{4_1} K_{3_2} K_{3_1}  \big)^{-1} \Big) \\
	& \quad =c^{-1} K = c^{-1} \veps_0 (V_1(x)).
\end{align*} 

Next we show relation (4) hold for $k=1$. It can be shown similarly that it holds for $k = 3,4,5,6$. We need to show that
$$e_0^{c_1} e_1^{c_2} (V_1(x)) = e_1^{c_2} e_0^{c_1} (V_1(x)).$$
Set $e_1^{c_2} (V_1(x)) = V_1(z)$, $e_0^{c_1} (V_1(z)) = V_1(z')$, $e_0^{c_1} (V_1(x)) =V_1(x')$ and  $e_1^{c_2}  (V_1(x')) =V_1(u)$. We have to show that 
$$z_m^{(l)'} = u_m^{(l)}$$
for $(l,m) \in\{(1,1), (1,2), (1,3), (1,4), (1,5), (1,6), (2,2), (2,3), (2,4), (2,5), (2,6), \\(3,3),(3,4), (3,5), (4,4)\}$. Since in this case $z_1^{(1)} = c_2x_1^{(1)}$ and $z_m^{(l)} = x_m^{(l)}$ for $(l,m) \not= (1,1)$ we have 
$K_z=K_x, K_{z 3_1} = K_{x 3_1}, K_{z 3_2} = K_{x 3_2}, K_{z 4_1} 
= K_{x 4_1}, K_{z 4_2} = K_{x 4_2}, K_{z 4_3} = K_{x 4_3}, K_{z 5_1} = K_{x 5_1}, K_{z 6_1} = K_{x 6_1}$ and $K_{z 6_2} = K_{x 6_2}$. Hence we have the following.
\begin{itemize}
\item $z_1^{(1)'} = \frac{z_1^{(1)}}{c_1} = \frac{c_2 x_1^{(1)}}{c_1} = c_2 x_1^{(1)'} = u_1^{(1)}$.
\item $z_2^{(1)'} = \frac{z_2^{(1)}}{c_1} = \frac{x_2^{(1)}}{c_1} = x_2^{(1)'} =  u_2^{(1)}$.
\item $z_2^{(2)'} = \frac{z_2^{(2)}}{c_1} = \frac{x_2^{(2)}}{c_1} = x_2^{(2)'} =  u_2^{(2)}$.
\item $z_3^{(1)'} = z_3^{(1)} \frac{K_z}{K_{z 3_1}} =  x_3^{(1)} \frac{K_x}{K_{x 3_1}} = x_3^{(1)'} = u_3^{(1)}$.
\item $z_3^{(2)'} = z_3^{(2)} \frac{K_{z 3_1}}{c_1 K_{z 3_2}} =  x_3^{(2)}  \frac{K_{x 3_1}}{c_1 K_{x 3_2}} = x_3^{(2)'} = u_3^{(2)}$.
\item $z_3^{(3)'} = z_3^{(3)} \frac{K_{z 3_2}}{c_1 K_z} =  x_3^{(3)} \frac{K_{x 3_2}}{c_1 K_x} = x_3^{(3)'} = u_3^{(3)}$.
\item $z_4^{(1)'} = z_4^{(1)} \frac{K_z}{K_{z 4_1}} =   x_4^{(1)} \frac{K_x}{K_{x 4_1}} = x_4^{(1)'} = u_4^{(1)}$.
\item $z_4^{(2)'} = z_4^{(2)} \frac{K_z K_{z 4_1}}{K_{z 4_2}} =   x_4^{(2)} \frac{K_x K_{x 4_1}}{K_{x 4_2}} = x_4^{(2)'} = u_4^{(2)}$.
\item $z_4^{(3)'} = z_4^{(3)} \frac{K_{z 4_2}}{c_1 K_z K_{z 4_3} } =  \frac{K_{x 4_2}}{c_1 K_x K_{x 4_3} }= x_4^{(3)'} = u_4^{(3)}$.
\item $z_4^{(4)'} = z_4^{(4)} \frac{K_{z 4_3}}{c_1 K_z} =   x_4^{(4)} \frac{K_{x 4_3}}{c_1 K_x} = x_4^{(4)'} = u_4^{(4)}$.
\item $z_5^{(1)'} = z_5^{(1)} \frac{K_z}{K_{z 5_1}} =  x_5^{(1)} \frac{K_x}{K_{x 5_1}} = x_5^{(1)'} = u_5^{(1)}$.
\item $z_5^{(2)'} = z_5^{(2)} \frac{K_{z 5_1}}{c_1 K_z} =  x_5^{(2)} \frac{K_{x 5_1}}{c_1 K_x} = x_5^{(2)'} = u_5^{(2)}$.
\item $z_6^{(1)'} = z_6^{(1)} \frac{K_z}{K_{z 6_1}} =  x_6^{(1)} \frac{K_x}{K_{x 6_1}} = x_6^{(1)'} = u_6^{(1)}$.
\item $z_6^{(2)'} = z_6^{(2)}  \frac{K_{z 6_1}}{K_{z 6_2}} =  x_6^{(2)} \frac{K_{x 6_1}}{K_{x 6_2}} = x_6^{(2)'} = u_6^{(2)}$.
\item $z_6^{(3)'} = z_6^{(3)} \frac{K_{z 6_2}}{c_1 K_z} =  x_6^{(3)} \frac{K_{x 6_2}}{c_1 K_x} = x_6^{(3)'} = u_6^{(3)}$.
\end{itemize}
since $u_1^{(1)} = c_2x_1^{(1)'}$ and $u_m^{(l)} = c_2x_m^{(l)'}$ for $(l,m) \not= (1,1)$.
Finally to show relation (5) we observe that $\bar{e_{6}}^{c_1}  \bar{e_4}^{c_1c_2} \bar{e_{6}}^{c_2}=\bar{e_4}^{c_2} \bar{e_6}^{c_1c_2} \bar{e_4}^{c_1}$ since $\mathcal{V}_2$ is a $\mathfrak{g}_1$-geometric crystal. Hence by Proposition \ref{relation2}, we have 
\begin{align*}
e_0^{c_1} e_2^{c_1c_2}e_0^{c_2} &= \bar{\sigma}^{-1} \bar{e_{6}}^{c_1}  \bar{\sigma} e_2^{c_1c_2} \bar{\sigma}^{-1} \bar{e_{6}}^{c_2}  \bar{\sigma} \\
&= \bar{\sigma}^{-1} \bar{e_{6}}^{c_1}  \bar{e_4}^{c_1c_2} \bar{e_{6}}^{c_2}  \bar{\sigma} \,
= \bar{\sigma}^{-1} \bar{e_4}^{c_2} \bar{e_6}^{c_1c_2} \bar{e_4}^{c_1}  \bar{\sigma} \\
&= e_2^{c_2} \bar{\sigma}^{-1}  \bar{e_6}^{c_1c_2} \bar{\sigma} e_2^{c_1}  \,
=e_2^{c_2} e_0^{c_1c_2}e_2^{c_1},
\end{align*}
which completes the proof.
\end{proof}

\bibliographystyle{amsalpha}

\end{document}